\documentclass[a4paper,11pt, twoside]{report}

\usepackage[english]{babel}
\usepackage[utf8]{inputenc}
\usepackage[T1]{fontenc}

\usepackage[a4paper,top=1in,bottom=1in,left=1in,right=1in,marginparwidth=1.75cm]{geometry}

\usepackage{afterpage}
\usepackage{amsmath}
\usepackage{amsthm}
\usepackage{amssymb}
\usepackage{csquotes}
\usepackage{enumitem}
\usepackage{graphicx}
\usepackage{lipsum}
\usepackage{booktabs}
\usepackage{mathrsfs}
\usepackage{listings}

\usepackage{titlesec}

\usepackage{caption}
\usepackage[ruled,vlined]{algorithm2e}
\makeatletter
\renewcommand{\SetKwInOut}[2]{%
  \sbox\algocf@inoutbox{\KwSty{#2}\algocf@typo:}%
  \expandafter\ifx\csname InOutSizeDefined\endcsname\relax
    \newcommand\InOutSizeDefined{}\setlength{\inoutsize}{\wd\algocf@inoutbox}%
    \sbox\algocf@inoutbox{\parbox[t]{\inoutsize}{\KwSty{#2}\algocf@typo:\hfill}~}\setlength{\inoutindent}{\wd\algocf@inoutbox}%
  \else
    \ifdim\wd\algocf@inoutbox>\inoutsize%
    \setlength{\inoutsize}{\wd\algocf@inoutbox}%
    \sbox\algocf@inoutbox{\parbox[t]{\inoutsize}{\KwSty{#2}\algocf@typo:\hfill}~}\setlength{\inoutindent}{\wd\algocf@inoutbox}%
    \fi%
  \fi
  \algocf@newcommand{#1}[1]{%
    \ifthenelse{\boolean{algocf@inoutnumbered}}{\relax}{\everypar={\relax}}%
    {\let\\\algocf@newinout\hangindent=\inoutindent\hangafter=1\parbox[t]{\inoutsize}{\KwSty{#2}\algocf@typo:\hfill}~##1\par}%
    \algocf@linesnumbered
  }}%
\makeatother

\usepackage{bm}
\usepackage[normalem]{ulem}

\usepackage[colorinlistoftodos]{todonotes}
\usepackage[colorlinks=true, allcolors=blue]{hyperref}

\usepackage[numbers, comma, square, sort&compress]{natbib}
\usepackage{tocloft}

\newcommand{\reporttitle}{A Calabi-Yau Metric on the Kummer Surface} 
 
\newcommand{\HRule}{\rule{\linewidth}{0.5mm}}
\begin{document}

\makeatletter
\patchcmd{\chapter}{\if@openright\cleardoublepage\else\clearpage\fi}{}{}{}
\makeatother

\begin{centering}

\makeatletter
\HRule \\[0.5cm]
{ \huge \bfseries \reporttitle}\\[0.5cm] 
\HRule \\[1cm]
 

   \large{Benjamin Shackleton}

{\large \today}\\[1cm]


\textbf{Abstract}\\

\end{centering}
We prove the existence of a Ricci flat metric on the Kummer K3 surface. The proof follows the general strategy of Donaldson's gluing construction, \cite{Simpap}. However, we tackle the analysis without appealing to weighted norms or conformal transformations to model spaces, instead relying solely on compact elliptic theory on usual H\"older and Sobolev spaces. As the Eguchi-Hanson space plays a central role in the construction, we also present and compare different descriptions of this space, showing explicitly that it is isometric to a suitable Gibbons-Hawking ansatz, \cite{GH1979}.

\tableofcontents
\vspace{-1cm}
\newtheorem{theorem}{Theorem}[chapter]

\newtheorem{lem}[theorem]{Lemma}
\newtheorem{prop}[theorem]{Proposition}
\newtheorem{defin}[theorem]{Definition}
\newtheorem{cor}{Corollary}[theorem]
\renewcommand{\thecor}{\thetheorem.\arabic{cor}}

\chapter{Introduction}

In 1976 Yau provided the first proof of the Calabi conjecture, showing that for a compact K\"ahler manifold, with first Chern class $\rho$, there exists a unique K\"ahler metric, $\omega$, within a given K\"ahler class, such that $\text{Ric}(\omega)=\rho$ for all $\rho\in c_1(M)\in H^2(M)$. This result was published in 1978, \cite{Yau1978}, and would be the work which earned Yau his Fields Medal in 1982. Due to it's many applications in geometry and physics, the case we are interested in is the Calabi-Yau, or Ricci flat, case where we have $\text{Ric}(\omega)=0$. In complex dimension 2 the simplest examples are K3 surfaces. In particular, the resolution of the Kummer surface, the manifold we will analyse in this approach. \\
\\
The construction is broken down into three parts, namely: Background, The Eguchi-Hanon Space and A Calabi-Yau Metric on the Kummer Surface. The first of these sections aims to familiarize the reader with essential background, in particular properties of Calabi-Yau 4-manifolds and stating the analytic prerequisites required for the later construction. In particular, we derive the complex Monge-Amp\`ere equation, given below, which will be essential for framing the analytic problem. Following this, we also give a summary of essential results regarding Sobolev spaces, H\"older spaces and elliptic theory.
\begin{center}
    $(\omega_0+2i\bar\partial\partial\phi)^2=\lambda\chi\wedge\bar\chi$
\end{center}
In the second section, we provide multiple formulations for the Eguchi-Hanson space. This is essential as much of our analysis will revolve around deeply understanding this space. It is by using this space that we resolve the singularities which occur while constructing the Kummer surface. The formulation that will prove to be most useful is the one from the original 1978 paper, \cite{EHmetric}, by Eguchi and Hanson. We also provide a description of the same space via the Gibbons-Hawking ansatz, \cite{GH1979}, and show that these two descriptions are equivalent. This second description is for interest and plays no part in the analysis which follows.\\
\\
In the final section, we construct both the Kummer surface and our Calabi-Yau metric on this surface, as was first done by Topiwala, \cite{Topiwala} and then later both both Donaldson, \cite{Simpap} and Jiang \cite{jiang202}. The main new results in this article are found in this construction as we produce this result without conformal transformations to model spaces, as was done by Donaldson \cite{Simpap}, and also without weighted spaces, as was done by Jiang \cite{jiang202}. To construct the Kummer surface we take $\mathbf{T}^4/\{\pm1\}$ and resolve the singularities by 'replacing' them with copies of the Eguchi-Hanson space. Following this, we produce an approximate solution and then perturb it; this gives us a complex Monge-Amp\`ere equation to solve. Then by careful application of compact elliptic theory and precise choices of suitable H\"older and Sobolev spaces, we can find estimates for the terms in our Monge-Amp\`ere equation. Using these estimates, we then show that we can apply the Banach fixed point theorem to give us our final result. \\

\textbf{Acknowledgments:} I want to thank Daniel Platt for first suggesting that I study the Kummer surface and also for his support while I developed this paper from that initial investigation. I also want to thank Simon Donaldson for his suggestions and feedback as I came to finalise this paper. Finally, I want to thank my friends and family for their support while I completed this work part time and away from a formal university setting.

\chapter{Background}

\section{Calabi-Yau 4-Manifolds}

In order to construct a Calabi-Yau metric we have to first describe what properties such a metric must satisfy. In this section we provide a definition for a real 4 dimensional manifolds to be Calabi-Yau, \ref{CYdef}. We conclude by formulating the complex Monge-Amp\`ere also known as the Calabi-Yau equation, \ref{CYequation}, which will be the foundation for our analysis in later chapters.

\begin{defin}
    Take $M$ to be a compact Riemannian manifold with Riemannian metric $g$ and Levi-Civita connection $\nabla$ and $P_\gamma:T_pM\rightarrow T_pM$ parallel transport along a curve $\gamma$. We define the \textbf{Holonomy group} of $(M,g)$ at $p$, written $\text{Hol}_p(g)$, to be:
    \begin{align}
        \text{Hol}_p(g)=\{ P_\gamma : \gamma:[0,1] \rightarrow M \text{ piecewise smooth curve with } \gamma(0)=\gamma(1) \}
    \end{align} \cite[p.213, p.6]{Cman, introCY}
\end{defin}
\textbf{Note:} The definitions of parallel transport and Levi-Civita connection used above have been omitted as we are assuming them as standard, for reference these definitions can be found in \cite{doca} if required.

\begin{defin} \label{CYdef}
    A 4-dimensional K\"ahler manifold is \textbf{Calabi-Yau} if $\text{Hol}(g)\subseteq SU(2)$, where $g$ is the associated Riemannian metric. \cite[p.7]{introCY}
\end{defin}

\textbf{Note:} This is not the only definition for Calabi-Yau manifolds for more information, see \cite{introCY}. The below proposition which we won't prove also could be used as a equivalent definition

\begin{prop} \label{Ballmann}
    For a K\"ahler manifold of complex dimension $m$, $\text{Hol}(g)\subseteq SU(m)$ iff the canonical bundle, $K_M=\Lambda^{(m,0)}T^*M$, admits a nowhere vanishing parallel holomorphic $(m,0)-$ form. \cite[p.57]{Ballmann}
\end{prop}

\begin{defin}
    For $(M,g)$ a real 4-dimensional Riemannian manifold, if $g$ is K\"ahler with respect to 3 different complex structures $I,J,K$ and these structures satisfy $IJ=-K$, $JK=-I$ and $KI=-J$. Then we call $I,J,K$ a \textbf{hyper-K\"ahler triple} or a \textbf{hyper-K\"ahler structure}. \cite[p.149]{JoyceBook}
\end{defin}

\begin{theorem}\label{HperkIJK}
    A real 4-dimensional manifold, $M$, with a hyper-K\"ahler structure, $I,J,K$, has $\text{Hol}(g)\subseteq SU(2)$, so $M$ is Calabi-Yau.
\end{theorem}
\textbf{Proof:} Denote the K\"ahler forms associated to almost complex structures $I,J$ and $K$ by $\omega_I, \omega_J$ and $\omega_K$. So for $X,Y\in TM$ and g the Riemannian metric we have:
\begin{align}
    \omega_I(X,Y)=g(IX, Y)
\end{align}
and analogous results for $J$ and $K$. Now taking $\nabla$ to be the Levi-Civita connection and $X,Y,Z\in TM$ we get the following relationship:
\begin{align}
    (\nabla_X\omega_I)(Y,Z)=X(\omega_I(Y,Z))-\omega_I(\nabla_X Y,Z)-\omega_I(Y,\nabla_X Z)=g((\nabla_XI)Y, Z)
\end{align}
Now consider the following, recalling that $\omega_I$ is closed:
\begin{align}
    \begin{split}
    0&=d\omega_I(X,Y,Z)=(\nabla_X\omega_I)(Y,Z)+(\nabla_Y\omega_I)(Z,X)+(\nabla_Z\omega_I)(X,Y)\\
    &=g((\nabla_XI)Y,Z)+g((\nabla_YI)Z,X)+g((\nabla_ZI)X,Y)\\
   &=g((\nabla_XI)Y,Z)-g((\nabla_YI)X,Z)+g((\nabla_ZI)X,Y)\\
   &=g((\nabla_XI)Y,Z)-g((\nabla_XI)Y,Z)+g((\nabla_XI)Z,Y)\\
   &=g((\nabla_XI)Z,Y)
    \\\implies \nabla_XI&=0, \forall X\in TM
    \end{split}
\end{align}
To manipulate the above expression I used that $I$ is integrable so we have $(\nabla_XI)Y=(\nabla_Y I)X$ and that $g((\nabla_X I)Y,Z)=-g((\nabla_XI)Z,Y)$. Observe that this clearly applies to the other complex structures, so $\nabla I=\nabla J=\nabla K=0$, this then gives us $\nabla\omega_I=\nabla\omega_J=\nabla\omega_K=0$, so these are parallel 2-forms. These forms are preserved by parallel transport and so also by the holonomy representation. Now fix a point, $p\in M$, and take $A\in \text{Hol}_p(g)$. As we have shown that $I,J, K$ are parallel, we have that:
\begin{align}
    I_p\circ A=A\circ I_p, J_p\circ A=A\circ J_p, K_p\circ A=A\circ K_p
\end{align}
Where $I_p, J_p, K_p$ are the complex structures at $p$. Due to the relationships between $I_p, J_p, K_p$, following from the relationships between $I, J, K$, we have a natural identification between $T_pM$ with the quaternions, $\textbf{H}$, taking $I_p, J_p$ and $K_p$ to $i, j$ and $k$ for a standard quaternion written as $q=a+bi+cj+dk, a,b,c,d\in \mathbf{R}$. So we have $\text{Hol}_p(g)$ is a subset of quaternions and by the orthogonality and the commutativity of the complex structures we have that $A$ identifies with a quaternion, $q\in \textbf{H}$, such that $|q|=1$, so we have:
\begin{align}
    \text{Hol}_p(g)\subseteq Sp(1)\cong SU(2)
\end{align}
Where $Sp(1)=\{q\in\textbf{H}:|q|=1\}$ and the isomorphism between $Sp(1)$ and $SU(2)$ is standard. From this the result follows. \qed

\begin{defin}
   Let $(X,g)$ be an n dimensional K\"ahler manifold, with metric $g$, the \textbf{Ricci curvature}, $\text{Ric}(u,v)$, is defined as $\text{Ric}(u,v)=\sum_{i=1}^ng(R(e_i, u)v, e_i)$. \cite[p.211]{Cman}
\end{defin}

\textbf{Note:} Above $R(u,v)$ denotes the curvature endomorphism, i.e: $R(u,v)=\nabla_u\nabla_v-\nabla_v\nabla_u-\nabla_{[u,v]}$.

\begin{defin}
  A manifold, $X$, is \textbf{Ricci flat} if $\text{Ric}(u,v)=0, \forall u,v\in T_pX$ at every $p\in X$. \cite[p.211]{Cman}
\end{defin} 

\begin{theorem}
    A simply connected K\"ahler manifold, $M$, of complex dimension $m$, is Calabi-Yau iff it is Ricci flat. \cite{CYric}
\end{theorem}
\textbf{Proof:} Recall the following $\text{Ric}(u,v)=\sum_{i=1}^ng(R(e_i, u)v, e_i)$ for $u,v\in T_pM$ we can then define the following $(1,1)$-form $\rho(u,v)=\text{Ric}(Iu,v)$, where $I$ is the complex structure as usual. So $\rho(u,v)=0$ iff $\text{Ric}(u,v)=0$, this is why $\rho$ is sometimes called the Ricci form. We are now going to make the following claim which I will prove later:
\begin{align}
    \rho=-i\partial\bar\partial \log\det(g_{i\bar{j}})
\end{align}
With this claim, we consider the canonical bundle, $K_M$. The Ricci form $\rho$ is the curvature of the canonical bundle $K_M$, this is a standard result, see \cite[p.203]{Cman}. If the holonomy is contained in $SU(m)$, then $K_M$ admits a parallel holomorphic volume form on $K_M$ by \ref{Ballmann}, $\Omega$. So we have the following, for a constant $k$:
\begin{align}
    \text{det}(g_{i\bar{j}})=k\|\Omega\|^2 \implies \rho=-i\bar\partial\partial \log (k\|\Omega\|^2)=0
\end{align}
So we have shown one direction of our claim. The other implication is as follows. $Hol_p(g) \subset U(n)$ by \cite[p.57]{Ballmann}. If $M$ is Ricci flat, then the curvature on the canonical bundle is zero. By the holonomy principle, a flat bundle over a simply connected manifold admits a parallel section. Hence, there exists a parallel volume form. So we have that $Hol_p(g)$  is contained in the stabiliser of the holomorphic volume form at $p$ within $U(n)$, which is $SU(n)$. All that remains is to show our claim. To show this, consider that we know that the curvature of a Hermitian line bundle with metric $h$ is given by $F=-\partial\bar\partial \log(h)$ \cite[p.186]{Cman} and we have that $\det(g_{ij})$ is a Hermitian metric on the canonical line bundle and that the Ricci curvature is the curvature of this line bundle. To see this, consider the following:
\begin{align}
    \begin{split}
    \rho(u,v)&=\sum_{i=1}^ng(R(e_i, Iu)v, e_i)=\sum_{i=1}^ng(J(R(e_i, u)v), e_i)=\sum_{i=1}^n\omega(R(e_i, u)v, e_i)\\&
    =i\sum_{i,j=1}^nR_{i\bar{j}}dz_i\wedge d\bar{z}_j=iF=-i\partial\bar\partial \log\det(g_{i\bar{j}})
    \end{split}
\end{align} \qed



\begin{theorem} \label{CYequation}
    A complex $2$-dimensional K\"ahler manifold, $M$, with K\"ahler metric, $\omega$. If $\omega$ is Calabi-Yau then the following equation holds, $\omega^2=\lambda\chi\wedge\chi$, where $\lambda>0$ and $\chi$ is a nowhere vanishing holomorphic 2-form.  \cite{Simpap}
\end{theorem}

\textbf{Proof:} First, we observe that both $\omega^2$ and $\chi\wedge\bar\chi$ are both volume forms for this manifold, so we can write:
\begin{align}
    \int_M\omega^2=\lambda\int_M\chi\wedge\bar\chi
\end{align}
As we want Ricci flat, we want that $\rho=-i\partial\bar\partial \log\det(g_{i\bar{j}})=0$. We also have that $\det(g_{i\bar{j}})=\omega^2$. Putting these together, we have:
\begin{align}
    \begin{split}
    0&=\rho=-i\partial\bar\partial\log\omega^2\\
    \implies \log(\omega^2)&=\log(\chi\wedge\bar\chi)+\text{constant}\\
    \implies \omega^2&=\lambda\chi\wedge\bar\chi
    \end{split}
\end{align}
\qed

\textbf{Note:} The equation $\omega^2=\lambda\chi\wedge\bar\chi$ is sometimes known as the \textbf{Calabi-Yau equation} or the \textbf{complex Monge-Amp\`ere equation}. \cite{Simpap}

\section{Analytic Prerequisites}

Here we provide the analytic setting for the construction, namely: Sobolev spaces, H\"older spaces and compact elliptic theory. We also state some known results which will prove essential going forward. We then provide a refinement of the Schauder estimate, \ref{improveSch}, in the particular setting which we will find ourselves when considering the Kummer surface.

\begin{defin}
    We define the \textbf{Sobolev norm}, $\|.\|_{L^q_k}$ on a Riemannian manifold $(M,g)$, to be:
    \begin{align}
        \|f\|_{L^q_k}=\left(\sum_{j=0}^k\int_M|\nabla^jf|^qdV_g\right)^{1/q}
    \end{align}
    where $dV_g$ is the volume form of the metric $g$. \cite[p.5]{JoyceBook}
\end{defin}

\begin{defin}
    We then define a \textbf{Sobolev space} on a Riemannian manifold $(M,g)$, $L^q_k(M)$ to be the completion of the space $C^\infty_c(M)$, or $C^\infty(M)$ for a compact manifold, with respect to the Sobolev norm. \cite[p.216]{Sobolev}\\
    \\
    There is also an equivalent definition using $\nabla^i$ to denote the weak derivative in the natural way on a compact Riemannian manifold $M$:
    \begin{align}
        L^q_k(M)=\{f\in L^q(M):|\nabla^jf|\in L^q(M), \forall j\leq k\}
    \end{align} \cite[p.5]{JoyceBook}
\end{defin}

\textbf{Note:} $L^q_K(M)$ is a Banach space with respect to the associated Sobolev norm.

\begin{defin}
    For $(M,g)$ a Riemannian manifold we define \textbf{C$^k$-spaces} $C^k(M)$, for $0\leq k\in\mathbf{Z}$, to be the space of continuous bounded functions with $k$ continuous bounded derivatives. On this space, we define the norm $\|.\|_{C^k(M)}$ to be:
    \begin{align}
        \|f\|_{C^k(M)}=\sum_{j=0}^k\sup_M|\nabla^j f|, 
    \end{align}
    where $\nabla$ is the Levi-Civita connection. \cite[p.5]{JoyceBook}
\end{defin}

\begin{defin}
    For $M$ a Riemannian manifold with Riemannian metric $g$ we define \textbf{H\"older spaces} $C^{k,\alpha}(M)$, $0\leq k\in\mathbf{Z}, \alpha\in(0,1)$ as the space of function $f\in C^k(M)$ for which the following supremum exists:
    \begin{align}
        [\nabla^k f]_\alpha=\underset{\underset{d(x,y)<\delta(g)}{x\neq y\in M}}{\sup}\frac{|\nabla^kf(x)-\nabla^kf(y)|}{d(x,y)^\alpha}
    \end{align}
    where $\delta(g)$ denotes the injectivity radius of $g$ over $M$ and we identify $\nabla^kf(x)$ with $\nabla^kf(y)$ for $x\neq y$ using parrallel transport along the unique geodesic from $x$ to $y$. This space equipped with the following \textbf{H\"older norm}:
    \begin{align}
        \|f\|_{C^{k,\alpha}(M)}=\|f\|_{C^k(M)}+[\nabla^kf]_\alpha
    \end{align}
    gives us our \textbf{H\"older space}. \cite[p.6]{JoyceBook}
\end{defin}

\begin{defin}
    A linear map, $P$, between sections of vector bundles $V$ and $W$ over a manifold, $M$, which we can write as:
    \begin{align}
        Pu=A^{i_1...i_k}\nabla_{i_1, ..., i_k}u+B^{i_1, ..., i_{k-1}}\nabla_{i_1, ..., i_{k-1}}u+...+K^{i_1}\nabla_{i_1}u+L(u)
    \end{align}
    where $A,B, ..., K$ are symetric tensors and $L$ is a real function on M. $P$ is an \textbf{elliptic operator} of rank $k$, if $\forall x\in M$ and each non-zero $\eta\in T_x^*M$ the map $\sigma_\eta(P,x)=A^{1, ..., k}\eta_{1}...\eta_k$ is invertible. Here $\eta_i$ are components of the covector in coordinates.
    \cite[p.9]{JoyceBook}
\end{defin}


\begin{theorem}
   The following are standard results which will be used later but which for the sake of brevity won't be proven here:
   \begin{enumerate}


       \item \label{Poincare} \textbf{Poincar\'e inequality}: For a compact manifold $M$, for $u\in C^2(M)$ such that $\int_M u=0$, there exists $c\in \mathbf{R}^+$ such that:
       \begin{align}
           \|u\|_{L^2}\leq c\|\nabla u\|_{L^2}
       \end{align}
       We also have $c= \lambda_1^{-1}$, where $\lambda_1$ is the minimal non-zero eigen value of the Laplacian.
    \cite[p.290]{Sobolev}
    \item \label{Bochner} \textbf{Integrated Bochner formula}: for a compact Riemannina manifold $(M,g)$ with non-negative Ricci curvature we have the following relationship:
    \begin{align}
        \begin{split}
        \int_M(\Delta u)^2 dV&=\int_M(|\nabla^2u|^2+\text{Ric}(\nabla u, \nabla u))dV\\
        \implies \int_M|\nabla^2u|^2dV&\leq\int_M(\Delta u)^2 dV
        \end{split}
    \end{align}

    \cite[p.338]{bochner}
    \item \label{LiYau} \textbf{Li-Yau estimate}: For a compact Riemannian manifold the first non-zero eigenvalue of the Laplacian are bounded by a constant which only depends on diameter and the Ricci curvature. \cite[Theorem 10]{LiYau}
    
    \item \label{schauder}\textbf{Schauder estimate}: Let $B_1 \subset B_2$ be balls in $\mathbf{R}^n$ and $\Delta$ be the Laplacian. Then for $\alpha\in(0,1)$, $f=\Delta u$ and a constant $C$ we have the following estimate:
    \begin{align}
        \|u\|_{C^{2,\alpha}(B_1)}\leq C(\|f\|_{C^{0,\alpha}(B_2)}+\|u\|_{C^0(B_2)})
    \end{align} \cite[p.15]{JoyceBook}

    \item \label{CZestimate} \textbf{Calderon-Zygmund}: For bounded open sets $\Omega'\subset \subset \Omega\subset \mathbf{R}^n$, $u\in L^p_2(\Omega)$ and $\Delta$ the Laplacian over $\Omega$, we have the following inequality where the constant $C$ depends on $\Omega, \Omega', n$ and $p$:
    \begin{align}
        \|u\|_{L^p_2(\Omega')}\leq C(\|\Delta u\|_{L^p(\Omega)}+\|u\|_{L^p(\Omega)})
    \end{align}
\cite[p.235]{CZestimate}
    
\end{enumerate}
\end{theorem}

\begin{prop} \label{improveSch}
    We have the following improvement on \ref{schauder} for compact Riemannian manifolds of dimension $n$, for all $p\in [1,\infty)$ and $\alpha\in (0,1)$:
    \begin{align}
        \|{u}\|_{C^{2,\alpha}(B_1)} \lesssim \|{\Delta u}\|_{C^{0,\alpha}(B_2)}+\|{u}\|_{L^p(B_2)}
    \end{align}
\end{prop}

\textbf{Proof:} Here we use the symbol, $\lesssim$, to denote the following $x\lesssim y\implies x\leq Ky$ for some constant $K$. To begin this proof, consider three balls of radii $1, \frac{3}{2}$ and $2$ so we have $B_1\subset B_{3/2}\subset B_2$. From our initial estimate \ref{schauder} we have:
\begin{align}
    \|u\|_{C^{2,\alpha}(B_1)}\lesssim\|\Delta u\|_{C^{0,\alpha}(B_{3/2})}+\|u\|_{C^0(B_{3/2})}
\end{align}
Here  Now by applying the Calderon-Zygmund estimate, \ref{CZestimate}, we have:
\begin{align}
    \|u\|_{L^p_2(B_{3/2})}\lesssim \|\Delta u\|_{L^p(B_2)}+\|u\|_{L^p(B_2)}
\end{align}
And for $p>n/2$ we have $\|u\|_{C^0}\leq C\|u\|_{L^2_p}$, this is the standard Sobolev embedding theorem. Putting this together, we get:
\begin{align}
    \begin{split}
    \|u\|_{C^{2,\alpha}(B_1)}&\lesssim \|\Delta u\|_{C^{0,\alpha}(B_{3/2})}+\|u\|_{C^0(B_{3/2})} \lesssim \|\Delta u\|_{C^{0,\alpha}(B_{3/2})}+\|\Delta u\|_{L^p(B_2)}+\|u\|_{L^p(B_2)}\\
    &\lesssim  \|\Delta u\|_{C^{0,\alpha}(B_{3/2})}+\|\Delta u\|_{C^{0,\alpha}(B_2)}+\|u\|_{L^p(B_2)}\lesssim \|\Delta u\|_{C^{0,\alpha}(B_2)}+\|u\|_{L^p(B_2)}
    \end{split}
\end{align}
Where the last inequality follows from $\|u\|_{L^p}\leq \|u\|_{C^0}\leq \|u\|_{C^{0,\alpha}}$. So we have the desired result \qed

\chapter{The Eguchi-Hanson Space}

\section{The Eguchi-Hanson Space in Spherical Coordinates}

When constructing the Kummer surface we will have to resolve a number of singular points; to do this we will find that replacing the region around them with the Eguchi-Hanson space is the correct way to proceed. So, it is essential to understand this space in detail. In this section we will formulate the Eguchi-Hanson in spherical coordinates, as given in its first conception \cite{EHmetric}, we also provide a calculation for the K\"ahler potential, \ref{KahlerPot}, of this space. This is notable as it differs by a normalisation constant from the well known potential provided in \cite[p.14]{Joyce2021} and \cite[p.160]{JoyceBook}.

\begin{defin}\label{EHspace}
    The following metric, defined on $\{r>a\}\times (S^3/\{\pm1\})$, is known as the \textbf{Eguchi-Hanson Metric}, $\omega_{EH}$, written in coordinates $(r, \theta,\phi,  \varphi): 0\le r, 0<\theta\le \pi, 0<\phi\le 2\pi, 0<\varphi\le 4\pi $ and a parameter $a:0<a<r$, we will call this space the \textbf{incomplete Eguchi-Hanson space}, $\bar{X}_{EH}$:
    \begin{align}
        ds^2=\left(1-\frac{a^4}{r^4}\right)^{-1}dr^2+\frac{r^2}{4}\left[\left(1-\frac{a^4}{r^4}\right)(d\varphi+\cos(\theta))^2+d\theta^2+\sin^2(\theta)d\phi^2\right]
    \end{align} 

It can also be written using the following: 
\begin{align}
    \begin{split}
    \sigma_1&=\frac{1}{2}(-\cos(\varphi)d\theta-\sin(\theta)\sin(\varphi)d\phi)\\
    \sigma_2&=\frac{1}{2}(\sin(\varphi)d\theta-\sin(\theta)\cos(\varphi)d\phi)\\
    \sigma_3&=\frac{1}{2}(-d\varphi -\cos(\theta)d\phi)\\
    ds^2&=\left(1-\frac{a^4}{r^4}\right)^{-1}dr^2+r^2\left(\sigma_1^2+\sigma_2^2+\left(1-\frac{a^4}{r^4}\right)\sigma_3^2\right)
    \end{split}
\end{align}
Using the substitution, $u^4=r^4-a^4$ with $u>0$, deals with the singularity at $r=a$ giving us a definition for the Eguchi-Hanson metric on $\{u>0\}\times (S^3/\{\pm 1\}) $.
    \cite{EHmetric}
\end{defin}

\begin{prop}
     We can write the Eguchi-Hanson metric as a power series in $a$, which to the first order is the flat metric.
\end{prop}
 \textbf{Proof:} First writing this metric as an expansion relies on the following well known result:
 \begin{align}
     \left(1-\frac{a^4}{r^4}\right)^{-1}=\sum_{n=0}^{\infty}\left(\frac{a^4}{r^4}\right)^n
 \end{align}
So we can write it in the following way:
\begin{align}
    ds^2=\sum_{n=0}^{\infty}\left(\frac{a^4}{r^4}\right)^n dr^2+\frac{r^2}{4}(\sigma_1^2+\sigma_2^2)+\frac{r^2}{4}\left(1-\frac{a^4}{r^4}\right)\sigma_3^2
\end{align}
This gives rise to the following approximation:
\begin{align}
    ds^2=\left(dr^2+\frac{r^2}{4}(\sigma_1^2+\sigma_2^2+\sigma_3^2)\right)+\left(\frac{a^4}{r^4}dr^2-\frac{a^4}{4r^2}\sigma_3^2\right)+\sum_{n=2}^\infty\left(\frac{a^4}{r^4}\right)^ndr^2
\end{align}
Looking at the first component of this expansion, we can identify this as the flat metric, as desired. \qed

As one can see, the space we have defined the Eguchi-Hanson metric on is incomplete. To complete this, we look at what happens as $r$ approaches our variable $a$ and the expansion given above provides the perfect tool to do this. Consider the following for $r^2=a^2+\rho^2$ with $\rho>0$ as $r$ is close to $a$, so $\rho$ is small:
\begin{align}
    \begin{split}
    ds^2&=\sum_{n=0}^{\infty}\left(\frac{a^4}{r^4}\right)^n dr^2+\frac{r^2}{4}(\sigma_1^2+\sigma_2^2)+\frac{r^2}{4}\left(1-\frac{a^4}{r^4}\right)\sigma_3^2\\&=\sum_{n=0}^\infty\left(\frac{a^4}{(a^2+\rho^2)^2}\right)^n\frac{\rho^2}{a^2+\rho^2}d\rho^2+\frac{a^2+\rho^2}{4}(\sigma_1^2+\sigma_2^2)+\frac{a^2+\rho^2}{4}\frac{2a^2\rho^2+\rho^4}{(a^2+\rho^2)^2}\sigma_3^2
    \end{split}
\end{align}
Now, considering as $\rho$ goes to $0$:
\begin{align}
    \begin{split}
        \lim_{\rho\rightarrow 0} ds^2\approx \lim_{\rho\rightarrow0}\left(\frac{a^2+\rho^2}{2a^2+\rho^2}d\rho^2+\frac{2\rho^2(a^2+\rho^2)}{4(a^2+\rho^2)}\sigma_3^2\right)+\frac{a^2}{4}(\sigma_1^2+\sigma_2^2)\approx \frac{1}{2}d\rho^2+\frac{\rho^2}{2}\sigma_3^2+\frac{a^2}{4}(\sigma_1^2+\sigma_2^2)
    \end{split}
\end{align}
We can now see, by considering $\rho$ going to zero, that the $\sigma_3$ term vanishes leaving $S^2$. So, to complete the Eguchi-Hanson space it is natural to add an $S^2$ at $r=a$.

\begin{defin}
    We define the \textbf{Eguchi-Hanson} space, $X_{EH}$, to be the metric completion of the space $\bar{X}_{EH}$, \ref{EHspace}, as given above above. So, concretely $\{r>a\}\times (S^3/\{\pm1\})$ with $S^2$ added at $r=a$.
\end{defin}

\begin{prop}\label{dsigmas}
    $\sigma_1, \sigma_2$ and $\sigma_3$ given above satisfy the following relationships:
    \begin{align}
        d\sigma_1=2\sigma_2\wedge \sigma_3, d\sigma_2=2\sigma_3\wedge\sigma_1, d\sigma_3=2\sigma_1\wedge\sigma_2
    \end{align}
\end{prop}
\textbf{Proof:} For brevity, I will only show the first of these relationships but the other two can be checked in an almost identical way:
\begin{align}
    \begin{split}
    d\sigma_1&=\frac{1}{2}d(-\cos(\varphi)d\theta-\sin(\theta)\sin(\varphi)d\phi)\\&=\frac{1}{2}(\sin(\varphi)d\varphi\wedge d\theta-\cos(\theta)\sin(\varphi)d\theta\wedge d\phi-\sin(\theta)\cos(\varphi)d\varphi\wedge d\phi)\\
    \sigma_2\wedge\sigma_3&=\frac{1}{4}(-\sin(\theta)d\theta\wedge d\varphi-\sin(\varphi)\cos(\theta)d\theta\wedge d\phi+\sin(\theta)\cos(\varphi)d\phi\wedge d\varphi)\\
    \implies d\sigma_1&=2\sigma_2\wedge\sigma_3
    \end{split}
\end{align}
\qed

When discussing this manifold it is often useful to use the following basis, and we will be using it freely moving forward:
\begin{align} \label{basis}
    e_0=\frac{dr}{\sqrt{1-\frac{a^4}{r^4}}}\text{ , } e_1=r\sigma_1\text{ , } e_2=r\sigma_2\text{ , } e_3=r\sqrt{1-\frac{a^4}{r^4}}\sigma_3
\end{align}
Observe this satisfies $ds^2=e_0^2+e_1^2+e_2^2+e_3^2$.

\begin{defin}\label{Hyperk}
    We can define three almost complex structures on this manifold:
    \begin{align}
    \begin{split}
    I(e_0)&=e_3, I(e_3)=-e_0, I(e_1)=e_2, I(e_2)=-e_1\\
    J(e_0)&=e_1, J(e_1)=-e_0, J(e_2)=e_3, J(e_3)=-e_2\\
    K(e_0)&=e_2, K(e_2)=-e_0, K(e_3)=e_1, K(e_1)=-e_3
    \end{split}
\end{align}
\end{defin}
\textbf{Note:} Observe it is simple to check that $I^2=-id, J^2=-id, K^2=-id$. It is also interesting to point out, these complex structures are not the only three possible complex structures. They have been chosen because of how nicely they interact with our chosen basis $\{e_i\}$.

\begin{prop}
    The K\"ahler forms associated with the complex structures $I,J$ and $K$ \ref{Hyperk} are:
    \begin{align}
        \begin{split}
        \omega_I&=rdr\wedge\sigma_3+r^2\sigma_1\wedge \sigma_2\\
        \omega_J&=\frac{r}{\sqrt{1-\frac{a^4}{r^4}}} dr\wedge \sigma_1+r^2\sqrt{1-\frac{a^4}{r^4}}\sigma_2\wedge\sigma_3\\
        \omega_K&=\frac{r}{\sqrt{1-\frac{a^4}{r^4}}}dr\wedge\sigma_2+r^2\sqrt{1-\frac{a^4}{r^4}}\sigma_3\wedge\sigma_1
        \end{split}
    \end{align}
    and $I,J,K$ form a hyper-K\"ahler tripple. 
\end{prop}

\textbf{Proof:} For simplicity, we will only show this for $\omega_I$, but it should be clear that it holds for $\omega_J$ and $\omega_K$. First, observe that by using the $e_i$'s, defined previously \ref{basis}, we can rewrite $\omega_I=e_0\wedge e_3+e_1\wedge e_2$. Recalling that a K\"ahler form, $\omega$, complex structure, $I$, and the Riemannian metric $g$ are related in the following way $\omega(u,v)=g(Iu,v)$. By applying this to our complex structure $I$, from \ref{Hyperk}, our Eguchi-Hanson metric $g$ and our $e_i$'s we have the following as the only two non-vanishing components:
\begin{align}
    \omega_I(e_0, e_3)=g(I(e_0), e_3)=g(e_3, e_3)=1, \omega_I(e_1, e_2)=g(I(e_1), e_2)=1
\end{align}
It follows that $\omega_I=e_0\wedge e_3+e_1\wedge e_2$ as desired. All that remains is to show that this is indeed a K\"ahler form. We can see immediately from the above argument that this form is positive, so all that is left is to show that it is closed. To do this, consider the following:
\begin{align}
    \begin{split}
    d(\omega_I)&=d(rdr\wedge\sigma_3+r^2\sigma_1\wedge \sigma_2)=d(rdr\wedge\sigma_3)+d(r^2\sigma_1\wedge\sigma_2)\\
    &=d(rdr)\wedge\sigma_3-rdr\wedge d\sigma_3+2rdr\wedge\sigma_1\wedge\sigma_2+r^2(d\sigma_1\wedge\sigma_2-\sigma_1\wedge d\sigma_2)
    \end{split}
\end{align}
We can now use the identities we found earlier, $d\sigma_i=2\epsilon_{ijk}\sigma_j\wedge\sigma_k$ \ref{dsigmas}, to get that the following holds:
\begin{align}
    \begin{split}
    d(\omega_I)&=-2rdr\wedge \sigma_1\wedge \sigma_2+2rdr\wedge\sigma_1\wedge\sigma_2+2r^2(\sigma_2\wedge\sigma_3\wedge\sigma_2-\sigma_1\wedge\sigma_3\wedge\sigma_1)\\&=-2rdr\wedge \sigma_1\wedge \sigma_2+2rdr\wedge\sigma_1\wedge\sigma_2=0
    \end{split}
\end{align}
So we have $\omega_I$ is closed. The final step to show that we have a hyper-K\"ahler triple is showing that the following relationships $IJ=-K, JK=-I, KI=-J$ hold. These are immediate from our definition of structures $I,J,K$ \ref{Hyperk} so we have a hyper-K\"ahler triple. 

\qed

\begin{cor} 
    The Eguchi-Hanson metric $\omega_{EH}$ is Calabi-Yau. \cite{Simpap, Joyce2021}
\end{cor}

\textbf{Proof:} This result follows immediately from the above proposition and \ref{HperkIJK} and will be important going forward. \qed

\begin{defin}
    A real valued function, $\varphi$, on a K\"ahler manifold $(X, \omega)$, which satisfies $\omega=i\partial\bar{\partial}\varphi$ is called a \textbf{K\"ahler potential}. \cite[p.50]{Cman}
\end{defin}

\begin{theorem} \label{KahlerPot}
    The K\"ahler potential for $\omega_I$, $\varphi_a(r)$, associated with the Eguchi-Hanson metric can be written as the following on $\{r>a\}\times (S^3/\{\pm1\})$:
    \begin{align}
         \varphi_a(r)=\frac{r^2}{2}+\frac{a^2}{4}\log \left(\frac{r^2-a^2}{r^2+a^2}\right)
    \end{align}
    
\end{theorem}

\textbf{Proof:} We will start by making the following claim which we will prove at the end:
\begin{align}
   \omega_I=-\frac{1}{2}d(I(d\varphi_a)
\end{align}
We already found $\omega_I=rdr\wedge \sigma_3+r^2\sigma_1\wedge \sigma_2=e_0\wedge e_3+e_1\wedge e_2$ above. So by finding the RHS of the above expression we can compare coefficients to find our K\"ahler potential. To begin, observe that our K\"ahler potential only depends on $r$. %
\begin{align}
    I(d\varphi(r))=I\left(\frac{\partial\varphi}{\partial r}dr\right)=I\left( \frac{\partial\varphi}{\partial r}\sqrt{1-\frac{a^4}{r^4}}e_0\right)=\frac{\partial\varphi}{\partial r}\sqrt{1-\frac{a^4}{r^4}}e_3
\end{align}
We can now apply $d$ to this once more and then compare:
\begin{align}
    \begin{split}
    d(I(d\varphi_a))&=d\left(\frac{\partial\varphi}{\partial r}\sqrt{1-\frac{a^4}{r^4}}e_3\right)=d\left(\frac{\partial\varphi}{\partial r}\sqrt{1-\frac{a^4}{r^4}}(\sqrt{1-\frac{a^4}{r^4}}r\sigma_3)\right)=d\left(\frac{\partial\varphi}{\partial r}\left(r-\frac{a^4}{r^3}\right)\sigma_3\right) \\
    &=d\left(\frac{\partial\varphi}{\partial r}\left(r-\frac{a^4}{r^3}\right)\right)\wedge\sigma_3-\frac{\partial\varphi}{\partial r}\left(r-\frac{a^4}{r^3}\right)d\sigma_3\\&
    =\left(\frac{\partial^2\varphi}{\partial r^2}\left(r-\frac{a^4}{r^3}\right)+\frac{\partial\varphi}{\partial r}\left(1+\frac{3a^4}{r^4}\right)\right)dr\wedge \sigma_3-2\frac{\partial\varphi}{\partial r}\left(r-\frac{a^4}{r^3}\right)\sigma_1\wedge\sigma_2
    \end{split}
\end{align}
By comparing the coefficient before the $\sigma_1\wedge\sigma_2$ we get the following, making sure to include the factor of $-\frac{1}{2}$ from our claim:
\begin{align}
    \begin{split}
    \frac{\partial\varphi_a}{\partial r}\left(r-\frac{a^4}{r^3}\right)&=r^2
    \implies \frac{\partial\varphi}{\partial r}=\frac{r^5}{r^4-a^4} \text{ , therefore}\\
    \varphi(r)&=\int\frac{r^5}{r^4-a^4}dr=\int\left(\frac{a^4r}{r^4-a^4}+r\right)dr=\frac{r^2}{2}+a^4\int\frac{1}{4v^2-a^4}dv\\&=\frac{r^2}{2}+a^4\int\frac{1}{2a^2(2v-a^2)}dv-a^4\int\frac{1}{2a^2(2v+a^2)}dv\\
    &=\frac{r^2}{2}+a^2\frac{\log (2v-a^2)}{4}-a^2\frac{\log (2v+a^2)}{4},
    \end{split}
\end{align}
where $v=\frac{r^2}{2}$ for ease when solving the integral. By combining the $\log$s we have the desired result if our initial claim holds. So now we need to show our claim, to do this consider a complex manifold with coordinated $z_i=x_i+iy_i$, note that we can write $d=\partial+\bar\partial$ and our complex structure acts in the following way $I(dx_i)=dy_i, I(dy_i)=-dx_i$. So for a function $f$ we get the following:
\begin{align}
    I(df)=I\left(\sum_i\left(\frac{\partial f}{\partial x_i}dx_i+\frac{\partial f}{dy_i}dy_i\right)\right)=\sum_i\left(\frac{\partial f}{\partial x_i}I(dx_i)+\frac{\partial f}{\partial y_i}I(dy_i)\right)=\sum_i\left(\frac{\partial f}{\partial x_i}dy_i-\frac{\partial f}{\partial y_i}dx_i\right)
\end{align}


We can then rewrite $I(df)$ as follows using known identities for $dx_i, dy_i, \frac{\partial}{\partial x_i}$ and $\frac{\partial}{\partial y_i}$:

\begin{align}
    \begin{split}
    I(df)&=\sum_i\left(\frac{\partial f}{\partial x_i}\frac{1}{2i}(dz_i-d\bar{z})-\frac{\partial f}{\partial y_i}\frac{1}{2}(dz_i+d\bar{z}_i)\right)\\
    &=\sum_i\left(dz_i\frac{1}{2}\left(-i\frac{\partial f}{\partial x_i}-\frac{\partial f}{\partial y_i}\right)+d\bar{z}_i\frac{1}{2}\left(i\frac{\partial f}{\partial x_i}-\frac{\partial f}{\partial y_i}\right)\right)\\
    &=i\sum_i\left(dz_i\frac{\partial f}{\partial z_i}-d\bar{z}_i\frac{\partial f}{\partial\bar{z}_i}\right)=i(\partial f-\bar\partial f)
    \end{split}
\end{align}
We can now consider the whole expression:
\begin{align}
    d(I(d))=i(\partial+\bar\partial)(\partial-\bar\partial)=i(\partial^2-\partial\bar\partial+\bar\partial\partial-\bar\partial^2)
\end{align}
Now using the following known identities $\partial^2=\bar\partial^2=0$ and $\bar\partial\partial=-\partial\bar\partial$, we get the desired result, where $\varphi$ and $\omega$ are the relevant K\"ahler potential and form:
$d(I(d\varphi_a))=-2\partial\bar{\partial}\varphi=-2\omega$. So we have that the claim is true and so our result holds. 
\qed

\begin{cor}
    Transforming from $r$ to $u$ by $r^4=u^4+a^4$, so from $\{r>a\}\times (S^3/\{\pm1\})$ to $\{u>0\}\times (S^3/\{\pm1\})$ gives the K\"ahler potential:
    \begin{align}
         \varphi_a(u)=\frac{\sqrt{u^4+a^4}}{2}+\frac{a^2}{4}\log \left(\frac{\sqrt{u^4+a^4}-a^2}{\sqrt{u^4+a^4}+a^2}\right)
    \end{align}
\end{cor}

\begin{cor}
   The K\"ahler potential we found above, $\varphi_a(u)$, differs from a standard K\"ahler potential found in the following, \cite[p.160]{JoyceBook}, by a factor of 2, $2\varphi_a(u)=\phi_a(u)$. So, they both define the same K\"ahler form up to our choice of normalisation for the $\sigma_i$'s. The potential, $\phi_a$, is given on $\{u>0\}\times (S^3/\{\pm1\})$ as $\phi_a(u)=\sqrt{u^4+a^4}-a^2\log\left(\frac{\sqrt{u^4+a^4}+a^2}{u^2}\right)$.
\end{cor} 

\textbf{Proof:} We begin with $\phi_a$ and use the following substitution, to simplify the expression, $s=\sqrt{u^4+a^4}$ $\implies u^2=\sqrt{s^2-a^4}$:
\begin{align}
    \begin{split}
    \phi_a(u)&=\sqrt{u^4+a^4}-a^2\log\left(\frac{\sqrt{u^4+a^4}+a^2}{u^2}\right)=s+a^2\log(u^2)-a^2\log(s+a^2)\\
    &=s+\frac{a^2}{2}\log(s^2-a^4)-a^2\log(s+a^2)=s+\frac{a^2}{2}\log((s-a^2)(s+a^2))-a^2\log(s+a^2)\\
    &=s+\frac{a^2}{2}(\log(s-a^2)-\log(s+a^2))=s+\frac{a^2}{2}\log\left(\frac{s-a^2}{s+a^2}\right)\\
    &=\sqrt{u^4+a^4}+\frac{a^2}{2}\log\left(\frac{\sqrt{u^4+a^4}-a^2}{\sqrt{u^4+a^4}+a^2}\right)=2\varphi_a(u)
    \end{split}
\end{align}
\qed

\section{The Gibbons-Hawking Ansatz}

In addition to the applications of the Eguchi-Hanson space in regards to the Kummer surface, it also comes with strong motivation from physics. Some of this physical importance can be seen in the original paper describing this space in 1978, \cite{EHmetric}. It is also seen by its equivalence to a Gibbons-Hawking anstaz with 2 centers, which we show explicitly below \ref{EHtoGH}.

\begin{defin}
     For $\mathbf{x}_i\in\mathbf{R}^3$, take $U=\mathbf{R}^3\backslash\{\mathbf{x}_1, ..., \mathbf{x}_k\}$, fix 'charges', $n_i\in\mathbf{Z}$, and fix $\epsilon\geq 0$. Then define the following harmonic function on $U$:
    \begin{align}
        V(\textbf{x})=\epsilon +\sum_i\frac{n_i}{|\mathbf{x}-\mathbf{x}_i|}
    \end{align}
    Let $\pi:M\rightarrow U$ be a principal $S^1$ bundle over $U$ with a connection 1-form $A$ satisfying $\nabla\times A=\nabla V$, where $\nabla\times$ and $\nabla$ are taken with respect to the Euclidean metric on $\mathbf{R}^3$, $d\mathbf{x}.d\mathbf{x}$. Then the \textbf{Gibbons-Hawking ansatz} is the following metric on $M$:
    \begin{align}
        ds^2=V(\textbf{x})^{-1}(d\psi+A)^2+V(\textbf{x})d\textbf{x}.d\textbf{x}
    \end{align} \cite{GH1979}
\end{defin}


\begin{theorem} \label{EHtoGH}
    The Gibbons-Hawking ansatz with 2 centers, of 'charge' $1$, and $\epsilon=0$ is isometric to the Eguchi-Hanson metric \ref{EHspace}, up to rescaling and choice of coordinates.
\end{theorem}
\textbf{Proof:} For simplicity we are going to take the 2 centers to lie on the z axis at $\pm \alpha$. So, we have $|\mathbf{x}-\mathbf{x}_i|^2=x^2+y^2+z_i^2=R_i^2$, where $z_1=z+\alpha$ and $z_2=z-\alpha$, $x$ and $y$ are the difference in the $x$ and $y$ coordinates. We can then rewrite the Gibbons-Hawking ansatz as:
 \begin{align}
     ds^2=\left(\frac{1}{R_1}+\frac{1}{R_2}\right)^{-1}(d\psi+A)^2+\left(\frac{1}{R_1}+\frac{1}{R_2}\right)d\mathbf{x}.d\mathbf{x}
 \end{align}
We now need to find an expression for $A$ and to do this we will exploit the relationship $\nabla\times A=\nabla V$, we will also swap to cylindrical coordinates, $(\rho, z, \varphi)$ here to make things simpler, note $R_i$'s will only be in terms of $\rho=\sqrt{x^2+y^2}$ and $z$, so we have that $R_1=\sqrt{\rho^2+(z+\alpha)^2}$ and $ R_2=\sqrt{\rho^2+(z-\alpha)^2}$ also take $\{e_\rho, e_z, e_\varphi\}$ as a orthonormal frame:
\begin{align}
\begin{split}
    \text{grad} V&=\text{grad}\left(\frac{1}{R_1}+\frac{1}{R_2}\right)= \partial_\rho\left(\frac{1}{R_1} +\frac{1}{R_2}\right)e_\rho+\partial_z\left(\frac{1}{R_1}+\frac{1}{R_2}\right)e_z\\
    &=-\left(\frac{\rho}{R_1^3}+\frac{\rho}{R^3_2}\right)e_\rho-\left(\frac{z_+}{R_1^3}+\frac{z_-}{R^3_2}\right)e_z
    \end{split}
\end{align}
As we have that $\nabla\times A=\nabla V$ we can write $A$ as a 1-form with only $d\varphi$, so we have that $A=A_\varphi(\rho, z)d\varphi$. This relationship then reduces to the following set of equations:
\begin{align}
-\partial_zA_\varphi=\partial_\rho V \text{ and }\frac{1}{\rho}\partial_\rho(\rho A_\varphi)=\partial_z V
\end{align}
Dealing with the first of these we integrate with respect to $z$ we will deal with each $\frac{\rho}{R^3_i}$ to make things clearer:
\begin{align}
    \int\frac{\rho}{R^3_i}dz=\int\frac{\rho}{(\rho^2+(z\pm \alpha)^2)^{3/2}}dz=\frac{z\pm \alpha}{\rho R_i}+C
\end{align}
Observe that by a suitable choice of gauge we can take $C=0$. This then gives us that:
\begin{align}
    A=\left(\frac{z+\alpha}{\rho R_1}+\frac{z-\alpha}{\rho R_2}\right)d\varphi
\end{align}
At this point, we should rewrite the Gibbons-Hawking ansatz in terms of the cylindrical coordinates we have been using:
\begin{align}
    ds^2=V^{-1}(d\psi+A)^2+V(d\rho^2+\rho^2 d\varphi^2+dz^2)
\end{align}
 Now most of the work is done and we have to take a couple of transformation to get to the familiar Eguchi-Hanson space. The first of these is to transform to a prolate spheroid with focuses at $\pm \alpha$ using coordinates $(\mu,\nu, \varphi)$, $\mu=\frac{R_1+R_2}{2\alpha}\geq 1$, $\nu=\frac{R_1-R_2}{2\alpha}\in [-1,1]$. Applying this transformation we have $V=\frac{2\mu}{\alpha(\mu^2-\nu^2)}$ and $A=\frac{2\nu(\mu^2-1)}{\rho(\mu^2-\nu^2)}d\varphi$. The last element to calculate is $d\mathbf{x}^2$ in these new coordinates, as this is routine but quite long it has been omitted:
 \begin{align}
     d\mathbf{x}^2=\alpha^2(\mu^2-\nu^2)\left(\frac{d\mu^2}{\mu^2-1}+\frac{d\nu^2}{1-\nu^2}\right)+\alpha^2(\mu^2-1)(1-\nu^2)d\varphi^2
 \end{align}
So our Gibbons-Hawking ansatz becomes:
\begin{align}
    ds^2=V^{-1}(d\psi+A)^2+V(\alpha^2(\mu^2-\nu^2)\left(\frac{d\mu^2}{\mu^2-1}+\frac{d\nu^2}{1-\nu^2}\right)+\alpha^2(\mu^2-1)(1-\nu^2)d\varphi^2)
\end{align}
We now have another series of transformations to get the familiar Eguchi-Hanson metric. Take $\mu=\cosh(u)=\frac{r^2}{2\alpha}$ and $\nu=\cos(\theta)$, $u\geq 0, \theta\in[0,\pi]$, these are taken as they give a useful pair of relationships, $\mu^2-1=\sinh^2(u)$ and $1-\nu^2=\sin^2(\theta)$, the variable $r$ is included as this makes things easier to write. Now, rewriting the metric with these variables gives us:
\begin{align}
    ds^2=&\frac{\alpha(\cosh^2u-\cos^2\theta)}{2\cosh u}\left(d\psi+\frac{2\cos\theta\sinh^2 u}{\cosh^2u-\cos^2\theta}d\varphi\right)^2\\&+2\alpha\cosh u\left(du^2+d\theta^2+\frac{\sinh^2 u\sin^2\theta}{\cosh^2u-\cos^2\theta}d\varphi^2\right)
\end{align}
Now recalling how we defined our $\sigma_i$'s in \ref{EHspace}:
\begin{align}
    \begin{split}
    \sigma_1&=\frac{1}{2}(-\cos(\varphi)d\theta-\sin(\theta)\sin(\varphi)d\phi)\\
    \sigma_2&=\frac{1}{2}(\sin(\varphi)d\theta-\sin(\theta)\cos(\varphi)d\phi)\\
    \sigma_3&=\frac{1}{2}(-d\varphi -\cos(\theta)d\phi)
    \end{split}
\end{align}
By expanding the above expression, putting $\sigma_i$'s into the expression and using $a^2=2\alpha$ we recover the familiar Eguchi Hanson metric \ref{EHspace}:
\begin{align}
    &ds^2=\left(1-\frac{a^4}{r^4}\right)^{-1}dr^2+\frac{r^2}{4}\left[\left(1-\frac{a^4}{r^4}\right)(d\varphi+\cos(\theta))^2+d\theta^2+\sin^2(\theta)d\phi^2\right]   
\end{align} \qed

\chapter{A Calabi-Yau Metric on the Kummer Surface}

\section{The Kummer Surface}

In this section, we provide a construction for the well known Kummer surface. Note, this is not the only known way to construct this surface; in fact, there are quite a number of different ways to construct this same surface. This construction will quotient $\mathbf{T}^4$ by $\mathbf{Z}_2$ and revolve the singularities by blowing them up and then identify these blow ups with the Eguichi-Hanson space we discussed in the previous section, \ref{EHspace}. \\
\\

Begin by considering the 4-torus $\mathbf{T}^4=\mathbf{C}^2/\Lambda$. We will then quotient by the following involution $\sigma:\mathbf{T}^4\rightarrow\mathbf{T}^4$ where $\sigma(z_1,z_2)=(-z_1, -z_2)$ for $ z_1, z_2\in \mathbf{C}$. 
\begin{prop}
    The map, $\sigma:\mathbf{T}^4\rightarrow\mathbf{T}^4$, has 16 fixed points, $p_1, ..., p_{16}$, and so the quotient $\mathbf{T}^4/\sigma$ is a manifold except for 16 singular points.
\end{prop}
\textbf{Proof:} A point $z\in \mathbf{C}^2$ is fixed iff $z=-z\pmod\Lambda$, so $2z\in\Lambda$. So we can write the set of fixed points as $\{z\in\mathbf{T}^4:2z=0 \text{ in } \mathbf{T}^4\}$. We can then just wite out the fixed points under this map:
\begin{align}
    \begin{split}
    &(0,0,0,0), \left(0,0,0,\frac{1}{2}\right), \left(0,0,\frac{1}{2},0\right), \left(0,\frac{1}{2},0,0\right),\left(\frac{1}{2},0,0,0\right)
    \left(0,0,\frac{1}{2},\frac{1}{2}\right), \left(0,\frac{1}{2},0\frac{1}{2}\right),\\ &\left(\frac{1}{2},0,0,\frac{1}{2}\right)
    \left(0,\frac{1}{2},\frac{1}{2},0\right),\left(\frac{1}{2},0,\frac{1}{2},0\right), \left(\frac{1}{2},\frac{1}{2},0,0\right), \left(0,\frac{1}{2},\frac{1}{2},\frac{1}{2}\right), \left(\frac{1}{2},0,\frac{1}{2},\frac{1}{2}\right),\left(\frac{1}{2}, \frac{1}{2}, 0, \frac{1}{2}\right),\\&\left(\frac{1}{2}, \frac{1}{2}, \frac{1}{2}, 0\right)
     \text{and } \left(\frac{1}{2}, \frac{1}{2}, \frac{1}{2}, \frac{1}{2}\right)
    \end{split}
\end{align}
  so we have 16 fixed points. When we quotient by a map with 16 fixed points this gives us the 16 singular points, $(p_1, ..., p_{16})$. \qed\\
\\

$\mathbf{T}^4/\sigma$ can be called an orbifold, as it is a topological space equipped with charts, much like a manifold except that each chart is modeled as $\mathbf{R}^n/G$ where $G$ is a finite group. Around each of the singular points we can model our orbifold, $\mathbf{T}^4/\sigma$, as $\mathbf{C}^2/\{\pm1\}$ with the singular point at zero, and the involution on this local model given by $\sigma_0$, this follows from our definition of $\sigma$. Now consider 16 disjoint regions, $\tilde{V}_1, ..., \tilde{V}_{16}$, each of which contains a singular point, each of these regions can be seen as $B_{p_i}(\frac{1}{9})/\{\pm1\}$. We are now going to apply a blow-up centered at the singular point in each of these regions.

\begin{prop}
    The map $\sigma_0$ lifts to a map $\bar{\sigma}$ in the blow-up, $Bl_0(\mathbf{C}^2)$. \label{prop1}
\end{prop}
\textbf{Proof:} By the definition of the blow-up map, found in \cite[p.98]{Cman}, we have that $Bl_0(\mathbf{C}^2)=\{((x,y), [u,v])\in\mathbf{C}^2\times\mathbf{CP}^1:xv=yu\}$. We can now define maps, $\bar{\sigma}:Bl_0(\mathbf{C}^2)\rightarrow Bl_0(\mathbf{C}^2)$ and $\pi:Bl_0(\mathbf{C}^2)\rightarrow \mathbf{C}^2$, in the following natural way:
\begin{align}
    \begin{split}
    \bar\sigma((x,y),[u,v])&=((-x, -y), [u,v])\\
    \pi((x,y), [u,v])&=(x,y)
    \end{split}
\end{align}
Considering the following:
\begin{align}
    \begin{split}
    \pi\circ\bar\sigma((x,y), [u,v])=(-x, -y)=\sigma_0\circ\pi((x,y),[u,v]), 
    \end{split}
\end{align}
which shows us that $\bar\sigma$ is a lift of $\sigma_0$ as desired. \qed
\begin{prop}
    The quotient $Bl_0(\mathbf{C}^2)/\langle\bar\sigma\rangle$ is a complex manifold of complex dimension 2. 
\end{prop}
\textbf{Proof:} First observe that the only condition which is not immediate is that the transition maps are holomorphic. Now, observe that $Bl_0(\mathbf{C}^2)$ is covered by two charts $U_1=\{((x,y),[u,v])\in Bl_0(\mathbf{C}^2):u\neq0\}$ and $U_2=\{((x,y),[u,v])\in Bl_0(\mathbf{C}^2):v\neq0\}$. Also, recall the definition for the map $\pi:Bl_0(\mathbf{C}^2)\rightarrow \mathbf{C}^2$ given in the proof of the last proposition. Taking a point, $((x,y),[u,v])$ in $U_1$ we can write it as $((x,y),[1,t])$ with $t=\frac{v}{u}$ and as we are in $Bl_0(\mathbf{C}^2)$ this implies $xt=y$, so $\pi((x,y), [u,v])=(x,xt)$. Similarly in $U_2$ we get that $\pi((x,y), [u,v])=(sy,y)$ where $s=\frac{u}{v}$. Now, define the following parameterizations $\phi_1:\mathbf{C}^2\rightarrow U_1$ and $\phi_2:\mathbf{C}^2\rightarrow U_2$ in the following way:
\begin{align}
    \phi_1(x,t)=((x,xt,[1,t]) \text{ and }\phi_2(s,y)=((sy,y),[s,1])
\end{align}
Now, working on the intersection, $U_1\cap U_2$, we can consider the transition functions $\phi_{12}$ and $\phi_{21}$ and calculate them directly as follows:
\begin{align}
    \begin{split}
    \phi_{12}(x,t)&=\phi_2^{-1}\circ\phi_1(x,t)=\phi^{-1}_2((x,xt),[1,t])=\phi_2^{-1}((x,xt),[\frac{1}{t},1])=(\frac{1}{t},xt)\\
    \phi_{21}(s,y)&=\phi_1^{-1}\circ\phi_2(s,y)=\phi_1^{-1}((sy,y),[s,1])=\phi_1^{-1}((sy,y),[1,\frac{1}{s}])=(sy,\frac{1}{s})
    \end{split}
\end{align}
Both of these transition functions are holomorphic on $U_1\cap U_2$ and so we have that $Bl_0(\mathbf{C}^2)$ is a complex manifold and by the prior proposition \ref{prop1} we can conclude that $Bl_0(\mathbf{C}^2)/\{\pm1\}$ is a complex manifold of complex dimension 2 as desired. \qed\\
\\
Taking each of the regions $\tilde{V}_1, ..., \tilde{V}_{16}$, which we recall locally looks like $\mathbf{C}^2/\{\pm 1\}$, and replacing them by a blow up centred at the respective singularity, we will have the Kummer surface. We will call each of these replacement regions, $V_1, ..., V_{16}$, with the numbering chosen arbitrarily, each of these can be locally viewed as $Bl_0(\mathbf{C}^2)/\{\pm 1\}$. For simplicity we will think of each of these $V_i$ as balls centred about the respective singularity, these balls should have radii \label{zeta}$\zeta=\frac{1}{9}$, chosen so that $V_i\cap V_j=\emptyset, \forall i\neq j$. This can clearly be done as each of our singularities are isolated.

\begin{defin}
    The \textbf{Kummer surface} can be defined as follows:
    \begin{align}
        Z=(T^4 \backslash  \{p_1, ..., p_{16}\})/\langle\sigma\rangle\cup_\pi(\bigcup_{j=1}^{16}V_j) 
    \end{align}
  Here $\cup_\pi$ denotes the disjoint union up to the following equivalence relationship. For $\pi:Bl_0(\mathbf{C}^2)\rightarrow \mathbf{C}^2$ as given prior. We can then define $\sim$ for $z\in\bigcup_{i=1}^{16}\tilde{V}_i\backslash\{p_1, ..., p_{16}\}$ and $y\in \bigcup_{i=1}^{16}V_i$, $z\sim y$ iff $\pi(y)=z$. 
\end{defin}
\textbf{Note:} This can be thought of as matching up each of the $\tilde{V}_i$ and $V_i$ except for the points $p_i$ and $\pi^{-1}(p_i)$ respectively. The notation $Z$ is chosen to agree with another example of the Kummer construction, namely \cite{Simpap}. If one wishes to reconcile our definitions with those used in Donaldson's article note he deals with just a single region with one singular point and then observes it easily extends to all 16 points. Also note, Donaldson uses the notation $X$ to denote $\mathbf{T}^4\backslash\{p_1, ..., p_{16}\}/\sigma$ and $Y$ to represent a single $V_i$.

\begin{cor}
    The Kummer surface is a complex 2 dimensional manifold. 
\end{cor}
\textbf{Proof:} This result follows directly from our prior proposition \ref{prop1}. \qed

\begin{prop}
    The Eguchi-Hanson space, $X_{EH}$, is diffeomorphic to $Bl_0(\mathbf{C}^2)/\{\pm1\}$. 
\end{prop}
To see this, we will split this problem into two cases, namely showing $Bl_0(\mathbf{C}^2\backslash\{\pi^{-1}(0)\})=\mathbf{C}^2\backslash\{0\}\cong \{r<a\}\times S^3$ and then the other case is immediate, as we have $Bl_0(\mathbf{C}^2)/(\mathbf{C}^2\backslash\{0\})=\mathbf{CP}^1\cong S^2$, these are the remaining unidentified parts from the first case. Consider the following to show the first case:
\begin{align}
    \tilde{P}&:\{r<a\}\times S^3\rightarrow \mathbf{C}^2/\{0\}\text{ ,  }\tilde{P}(r,u)=ru\\
    \tilde{P}^{-1}&:\mathbf{C}^2/\{0\}\rightarrow \{r<a\}\times S^3\text{ ,  }\tilde{P}^{-1}(z)=(|z|, z/|z|)
\end{align}
Both of these maps are diffeomorphisms so we have $Bl_0(\mathbf{C}^2/\{0\})\cong \{r<a\}\times S^3$. We can then define our diffeomorphism, $P:X_{EH} \rightarrow Bl_0(\mathbf{C}^2)/\{\pm1\}$, where this map identifies each case as outlined above, including the quotient by $\mathbf{Z}_2$ in the first case. \qed

\section{An Approximate Calabi-Yau Metric}

Now that we have an understanding of the Kummer surface, we want to produce an approximate Calabi-Yau metric. To do this we use that the flat metric on $\mathbf{T^4}$ and the Eguchi-Hanson metric are Calabi-Yau, so by introducing a gluing region as we vary smoothly between these two metrics. The resulting glued metric will be almost Calabi-Yau. We then also provide a nowhere vanishing holomorphic 2-form over the whole Kummer surface, this sets us up to produce a Calabi-Yau equation, \ref{CYequation}, which we will solve in the final section, \ref{4}.

\begin{prop}
    The canonical K\"ahler metric on $\mathbf{T}^4$ is the flat metric, with the flat K\"ahler form $\omega_{\mathbf{T}^4}=\frac{i}{2}(dz_1\wedge d\bar{z}_1+dz_2\wedge d \bar{z}_2)$.
\end{prop}
\textbf{Proof:} This follows from noting that the flat K\"ahler metric on $\mathbf{C}^2$ is translation invariant and so under quotient of $\Lambda$, to give $\mathbf{T}^4=\mathbf{C}^2/\Lambda$, it is preserved. Given the metric is preserved, the K\"ahler form is also clearly preserved. \qed

\begin{prop}
    The flat K\"ahler form induces a K\"ahler form on $Z\backslash\bigcup_{i=1}^{16}V_i$.
\end{prop}
\textbf{Proof:} To show this it will be sufficient to show that the K\"ahler form and metric on $\mathbf{T}^4$ and $\mathbf{T}^4/\sigma$ are the same away from the singular points. This result can be seen by the following observation: $d(\sigma(z_i))\wedge d(\sigma(\bar{z}_i))=d(-z_i)\wedge d(-\bar{z}_i)=dz_i\wedge d\bar{z}_i, \forall i\in\{1,2\}$ and so we see that both the K\"ahler form and metric are preserved under $\sigma$. \qed

From the general idea given at the start of this section, the reader can probably guess that defining a metric over the Kummer surface will require that we define a metric which 'connects' $\omega_{T^4}$ and $\omega_{EH}$. We will do this in the following way. Begin by defining some smooth cut off function, $\beta:\mathbf{R}\rightarrow [0,1]$, recalling that $\zeta$ is the diameter of the region $V_i$ around our singularity $p_i$, \ref{zeta}:  
\begin{align}
    \beta(t)=\begin{cases}
       1 & \forall t\leq\frac{\zeta}{4}\\
       \text{Smooth} & \frac{\zeta}{4}<t<\frac{\zeta}{2}\\
       0 & \forall t\geq\frac{\zeta}{2}
    \end{cases}
\end{align}
We then define the following using our K\"ahler potential, $\varphi_a$, \ref{KahlerPot}, and the flat K\"ahler potential in our coordinates $\varphi_{flat}=\frac{r^2}{2}$, $G=\varphi_a-\varphi_{flat}=\frac{a^2}{4}\log(\frac{r^2-a^2}{r^2+a^2})$. Now we define the following:
\begin{align}
\omega_{I,a}=\omega_{T^4}+i\partial\bar\partial((\beta\circ r)G)
\end{align}
Note that when $\beta(r)=0$ $\omega_{I,a}=\omega_{flat}=\omega_{T^4}$ and when $\beta(r)=1$ we have $\omega_{I,a}=\omega_{EH}$, so we can see that the metric varies smoothly from $\omega_{EH}$ to $\omega_{T^4}$ within each $V_i$. So, defining the metric around a singularity, $p_i$, we have:
\begin{align}
    \omega_0^i=\begin{cases}
        \omega_{T^4} & Z/V_i \\
        \omega_{I,a} & V_i
    \end{cases}
\end{align}
So to define a metric over all the $p_i$'s we need to define $r_i$ as the distance from singularity $p_i$ so we get the following metric defined over the whole of the Kummer surface,  $Z$: 
\begin{align} \label{omega0}
\omega_0=\omega_{T^4}+i\sum_{i=1}^{16}\partial\bar\partial((\beta\circ r_i)G)
\end{align}
This metric takes the form of the Eguchi-Hanson metric 'close' to each $p_i$ and the flat metric away from the $p_i$'s and varies smoothly between them within each of the $V_i$'s.

\begin{prop} \label{Omega}
    $\Omega=\omega_J+i\omega_K$ is a holomorphic $(2,0)$-form more particularly $\Omega=P^*(dz_1\wedge dz_2)$, where $P:X_{EH}\rightarrow \mathbf{C}^2/\{\pm1\}$. 
\end{prop}

\textbf{Note:} Below I slightly abuse notation by not including the map $P$, you can consider $P$ or $P^{-1}$ to be applied to whole expressions where appropriate. \\
\\

\textbf{Proof:} First, we will write $\omega_J$ and $\omega_K$ in terms of $u$ so we are considering them over the space $\{u>0\}\times (S^3/\{\pm 1\})$, recall $u^4=r^4-a^4\implies \frac{du}{dr}=\frac{r^3}{u^3}$. 
\begin{align}
    \begin{split}
    \omega_J&=\frac{r}{\sqrt{1-\frac{a^4}{r^4}}} dr\wedge \sigma_1+r^2\sqrt{1-\frac{a^4}{r^4}}\sigma_2\wedge\sigma_3=\frac{r^3}{\sqrt{r^4-a^4}} dr\wedge \sigma_1+\sqrt{r^4-a^4}\sigma_2\wedge\sigma_3\\
    &=udu\wedge\sigma_1+u^2\sigma_2\wedge\sigma_3 \text{ , and similarly we get: }\omega_K=...=udu\wedge\sigma_2+u^2\sigma_3\wedge\sigma_1\\
    \Omega&=\omega_J+i\omega_K=udu\wedge\sigma_1+u^2\sigma_2\wedge\sigma_3+i(udu\wedge\sigma_2+u^2\sigma_3\wedge\sigma_1)\\
    &=udu\wedge(\sigma_1+i\sigma_2)+u^2\sigma_3\wedge(i\sigma_1-\sigma_2)
    \end{split}
\end{align}
We now need to find expressions for $\sigma_i$'s and $du$ in terms of $dz_1, dz_2$ and their conjugates. The expression for $du$ follows directly from the following identity $u^2=|z_1|^2+|z_2|^2$. For the $\sigma_i$'s we write $z_1=u\cos(\frac{\theta}{2})e^{\frac{i}{2}(\varphi+\phi)}$ and $z_2=u\sin(\frac{\theta}{2})e^{\frac{i}{2}(\varphi-\phi)}$ and use this to write out $dz_1, dz_2, d\bar{z}_1$ and $d\bar{z}_2$ and then by comparing to out $\sigma_i$'s we have the desired expressions. As this is quite routine but messy it has been omitted. The results are the following 1-form relationships:
\begin{align}
    \begin{split}
    du&=\frac{1}{2u}(\bar{z}_1dz_1+z_1d\bar{z}_1+\bar{z}_2dz_2+z_2d\bar{z}_2)\\
    \sigma_1&=-\frac{1}{2u^2}(z_2dz_1-z_1dz_2+\bar{z}_2d\bar{z}_1-\bar{z}_1d\bar{z}_2)\\
    \sigma_2&=\frac{i}{2u^2}(z_2dz_1-z_1dz_2-\bar{z}_2d\bar{z}_1+\bar{z}_1d\bar{z}_2)\\
    \sigma_3&=-\frac{i}{2u^2}(\bar{z}_1dz_1-z_1d\bar{z}_1+\bar{z}_2dz_2-z_2d\bar{z}_2)
    \end{split}
\end{align}
We now want to consider $\sigma_1+i\sigma_2$:
\begin{align}
    \begin{split}
    \sigma_1+i\sigma_2&=-\frac{1}{2u^2}(z_2dz_1-z_1dz_2-\bar{z}_2d\bar{z}_1+\bar{z}_1d\bar{z}_2+z_2dz_1-z_1dz_2+\bar{z}_2d\bar{z}_1-\bar{z}_1d\bar{z}_2)\\
    &=\frac{1}{u^2}(z_1dz_2-z_2dz_1)
    \end{split}
\end{align}
We now also have $i\sigma_1-\sigma_2=\frac{i}{u^2}(z_1dz_2-z_2dz_1)$ by multiplying $\sigma_1+i\sigma_2$ by $i$. We can now consider $\Omega=\omega_J+i\omega_K$, so we have that:
\begin{align}
    \begin{split}
    \Omega&=\frac{1}{u}du\wedge(z_1dz_2-z_2dz_1)+\sigma_3\wedge(z_1dz_2-z_2dz_1)=(\frac{1}{u}du+i\sigma_3)\wedge(z_1dz_2-z_2dz_1)\\
    &=\frac{1}{2u^2}((\bar{z}_1dz_1+z_1d\bar{z}_1+\bar{z}_2dz_2+z_2d\bar{z}_2)+(\bar{z}_1dz_1-z_1d\bar{z}_1+\bar{z}_2dz_2-z_2d\bar{z}_2))\wedge(z_1dz_2-z_2dz_1)\\
    &=\frac{1}{u^2}(\bar{z}_1dz_1+\bar{z}_2dz_2)\wedge(z_1dz_2-z_2dz_1)=\frac{1}{u^2}(|z_1|^2dz_1\wedge dz_2-|z_2|^2dz_2\wedge dz_1)\\
    &=\frac{|z_1|^2+|z_2|^2}{u^2}dz_1\wedge dz_2=dz_1\wedge dz_2
    \end{split}
\end{align} 
All we have left to show is that $\Omega$ is a holomorphic $(2,0)$-form, to do this we need to re-introduce our map $P:X_{EH}\rightarrow Bl_0( \mathbf{C}^2)/\{\pm 1\}$. First, observe that this clearly takes our $(2,0)$-form $dz_1\wedge dz_2$ to a $(2,0)$-form, so all that remains to be shown is that $\bar\partial P=P\bar\partial$, i.e. $P$ is a holomorphic map. Consider the following:
\begin{align}
    P(u, \theta, \varphi, \phi)=(u\cos(\frac{\theta}{2})e^{\frac{i}{2}(\varphi+\phi)}\text{ , }u\sin(\frac{\theta}{2})e^{\frac{i}{2}(\varphi-\phi)})=(z_1, z_2)
\end{align}
This is a composition of holomorphic maps for $u>0$ so $P$ is holomorphic so we have that $\Omega$ is holomorphic by  $\bar\partial\Omega=\bar\partial(P(dz_1\wedge dz_2))=P(\bar\partial(dz_1\wedge dz_2))=0$ and we have that $\Omega$ is a holomorphic $(2,0)$-form.
\qed\\
\\

We can now define a holomorphic volume form over the whole of $Z$. To begin with, we will define it over $Z\setminus\{\pi^{-1}(p_i)\}_i$ and then we can apply Hartogs theorem, \cite[p.6]{Cman}, to extend it to the whole of $Z$.\\
\\
We start by recalling the holomorphic volume form on $\textbf{T}^4/\{\pm 1\}$, away from $\{p_i\}$, is just the standard one for $\textbf{T}^4$, so $dz_1\wedge dz_2$. From our proposition above, \ref{Omega}, we have $\Omega=dz_1\wedge dz_2$ on the Eguchi-Hanson space, so we have that these forms exactly agree. Considering, the punctured ball, $\mathbf{C}^2\setminus\{p_i\}$, is biholomorphic to the Eguchi-Hanson space with the exceptional 2-sphere removed, we can make this exact identification. This means $\Omega=dz_1\wedge dz_2$ is our holomorphic volume form on $Z\setminus\{\pi^{-1}(p_i)\}_i$. By now applying Hartogs theorem we can extend this over $\{\pi^{-1}(p_i)\}_i$ so we have the final holomorphic volume form. \\
\\


\section{Estimates for $\mathbf{e_a}$, the Laplace Operator and Q} 

Using our Calabi-Yau equation, \ref{CYequation}, and perturbing our approximate metric by $\bar\partial\partial\phi$, we can now reduce finding a Calabi-Yau metric to finding a function, $\phi$, satisfying a perturbed Calabi-Yau equation, \ref{CYequation0}. This equation breaks up into three parts: an unperturbed term, a linear term, and a quadratic term. We then present an estimate for each of these parts: \ref{ea}, \ref{laplace} and \ref{quadratic}.\\
\\

Taking $\mathcal{D}=2i\bar\partial\partial$, we can consider our approximate K\"ahler form, $\omega_0$, \ref{omega0}, and a function $\phi$. We then have $\omega_\phi=\omega_0+\mathcal{D} \phi$, is K\"ahler if $\omega_0+\mathcal{D} \phi>0$. So, to find our Calabi-Yau metric we need to solve the following Calabi-Yau equation, \ref{CYequation}, for $\phi$:
\begin{align}\label{CYequation0}
    (\omega_0+\mathcal{D} \phi)^2=\lambda\chi\wedge\bar\chi:\omega_0+\mathcal{D} \phi>0
\end{align}
For $\lambda>0$. We can then expand and rearrange this equation, splitting it into different parts to make the analysis easier:
\begin{align}
    \begin{split}
    &\omega_0^2-\lambda\chi\wedge\bar\chi+\omega_0\wedge\mathcal{D} f+\mathcal{D}\phi\wedge\omega_0+(\mathcal{D} \phi)^2=0\\
    &\tilde{e}_a=\omega_0^2-\lambda\chi\wedge\bar\chi\text{ , }  L(\phi)=2\omega_0\wedge\mathcal{D} \phi\text{ , }\tilde{Q}(\phi)=(\mathcal{D} \phi)^2\\
    &\tilde{e}_a+L(\phi)+\tilde{Q}(\phi)=0
    \end{split}
\end{align}
From here it is then convenient to divide by a factor of $\omega_0^2$ so we are considering functions and not $(2,2)$-forms. Doing this, we observe that $L/\omega_0^2=\Delta_{\omega_0}$, which is the familiar Laplacian, from now we will write this as $\Delta$. This gives us:
\begin{align}
    \begin{split}
    &e_a+\Delta(\phi)+Q(\phi)=0\\
    &e_a=(\omega_0^2-\lambda \chi\wedge\chi)/\omega_0^2\text{ , } \Delta \phi=2\omega_0\wedge \mathcal{D} \phi/\omega_0^2\text{ , } Q(\phi)=(\mathcal{D} \phi)^2/\omega_0^2
    \end{split}
\end{align}

We will now compute suitable estimates for $e_a$, $\Delta^{-1}$ and $Q$. To do this consider the following 3 metrics which will prove essential in the following analysis, $M$ is a compact manifold which is covered by balls, $B_r(p)$, with $r=\mathcal{O}(a)$:
\begin{align}
    \begin{split}
\|.\|_X&=a^{-4+\epsilon}\|.\|_{L^2_2(M)}+a^\alpha\|.\|_{C^{2,\alpha}(M)}\\
    \|.\|_Y&=a^{-4+\epsilon}\|.\|_{L^2(M)}+\|.\|_{C^{0,\alpha}(M)}\\
    \|.\|_{C^{k,\alpha}(M)}&=\max_{{p\in M}}\|.\|_{C^{k, \alpha}(B_r(p))}
    \end{split}
\end{align}
where $a$ is our familiar variable from the Eguchi-Hanson metric \ref{EHspace}, $\alpha\in(0,\frac{1}{3})$\label{alpha} and $\epsilon=2-\frac{n}{p}=\frac{4}{3}>0$, for $n=4$ and $p=6$, the choice of $p$ here is not unique and where the choice becomes important will be mentioned below. We then define the spaces $X$ and $Y$ as the sets of smooth, mean zero functions completed with respect to the metrics $\|.\|_X$ and $\|.\|_Y$ respectively. The norms, $\|.\|_X$ and $\|.\|_Y$, and the following analysis are similar to the ones used in \cite{Platt}. 

\begin{prop} \label{ea}
    We have the following inequality, for $c\in \mathbf{R}^+$ and $a$ the familiar constant from the Eguchi-Hanson space:
    \begin{align}
        \|e_a\|_Y\leq c a^\epsilon
    \end{align}
\end{prop}

\textbf{Proof:} Consider the following which we know to be true as the Eguchi-Hanson space and $\mathbf{T}^4$ are Calabi-Yau also note that for $\lambda>0$ :
\begin{align}
    \begin{split}
    \omega_{EH,a}^2-\lambda\chi\wedge\bar\chi&=0 \text{, on the Eguch-Hanson space}\\
    \omega_{flat}^2-\lambda\chi\wedge\bar\chi&=0\text{, on }\mathbf{T}^4
    \end{split}
\end{align}
Now by considering the explicit formulas for $\omega_{EH,a}$ and $\omega_{flat}$ given earlier, in the gluing region we get that:
\begin{align}
     \omega_{EH,a}^2-\omega_{flat}^2=\mathcal{O}(a^4)
\end{align}
So we have that $\omega_0^2-\lambda\chi\wedge\bar\chi=0$ everywhere except the gluing region. We can now rewrite this in the following way:
\begin{align}
    \omega_0^2-\lambda\chi\wedge\bar\chi=(\omega_{flat}^2-\lambda\chi\wedge\bar\chi)+(\omega_0^2-\omega_{flat}^2)=0+\mathcal{O}(a^4)
\end{align}

We can now consider $\|e_a\|_Y$ directly:

\begin{align}
    \begin{split}
    \|e_a\|_Y&=a^{-4+\epsilon}\|e_a\|_{L^2(Z)}+\|e_a\|_{C^{0,\alpha}(Z)}=a^{-4+\epsilon}\|e_a\|_{L^2(Z)}+\|e_a\|_{C^0(Z)}+[e_a]_\alpha\\
    &=a^{-4+\epsilon}\mathcal{O}(a^4)+\mathcal{O}(a^4)+\mathcal{O}(a^4)d(x,y)^{-\alpha}=\mathcal{O}(a^\epsilon)\leq c a^\epsilon 
    \end{split}
\end{align}
for small $a$.
\qed

\textbf{Note:} There is an assumption in the above argument which requires some justification, namely that the two $\lambda$'s in $\omega_{EH,a}^2-\lambda\chi\wedge\bar\chi=0$ and $\omega_{flat}^2-\lambda\chi\wedge\bar\chi=0$ are the same. This follows from the construction, as in general there is no reason these should be equal in two separate spaces but in our construction we require that we rescale the Eguchi-Hanson space to 'fit' into the region $V_i$, this rescaling is such that we have $G=\varphi_a-\varphi_{flat}=\frac{a^2}{4}\log(\frac{r^2-a^2}{r^2+a^2})$, which then gives us that $|\omega_{EH,a}^2-\omega_{flat}^2|=\mathcal{O}(a^4)$ but if $\lambda_{EH}\neq\lambda_{flat}$ we also have $|\omega_{EH,a}^2-\omega_{flat}^2|=\mathcal{O}(a^4)+|(\lambda_{EH}-\lambda_{flat})\chi\wedge\bar\chi|$, which is clearly contradictory.\\
\\
 Going forward, we are going to use the following symbol, $\lesssim$, to denote the following $x\leq Cy$, where $C$ is a constant independent of $a$, the familiar variable from the Eguchi-Hanson space. So, $x\lesssim y\implies x\leq Cy$.

\begin{prop} \label{laplace}
    The map $\Delta:X\rightarrow Y$ has a bounded inverse, specifically $\|\Delta^{-1}\|_{Y\rightarrow X}\leq 1$.
\end{prop}

\textbf{Proof:} In order to show this, it is sufficient to show that $\| u\|_X\lesssim\|\Delta u\|_Y$. Recall:
\begin{align}
    \|u\|_X=a^{-4+\epsilon}\|u\|_{L^2_2}+a^\alpha\|u\|_{C^{2, \alpha}}
\end{align}
We will deal with each part of this separately:
\begin{align}\label{L22}
    \begin{split}
    \|u\|^2_{L^2_2}&=\sum_{j=0}^2\|\nabla^ju\|^2_{L^2}=\|u\|^2_{L^2}+\|\nabla u\|^2_{L^2}+\|\nabla^2u\|^2_{L^2}\lesssim \|u\|^2_{L^2}+\|\nabla^2u\|^2_{L^2}\\&\underset{\ref{Bochner}}{\lesssim}\|u\|^2_{L^2}+\|\Delta u\|^2_{L^2}\underset{\ref{LiYau}}{\lesssim} \frac{1}{\lambda}\|\Delta u\|^2_{L^2} +\|\Delta u\|_{L^2}^2\lesssim \|\Delta u\|^2_{L^2}\\
    \implies a^{-4+\epsilon}\|u\|_{L^2_2}&\lesssim\|\Delta u\|_Y
    \end{split}
\end{align}
Observe that we can apply the Li-Yau estimate, \ref{LiYau}, because our construction has Ricci curvature and diameter bounded independently from $a$, so the first non-zero eigen value of the Laplacian, $\lambda$ is bounded independently of $a$. To see these two terms are bounded we can first consider the diameter, we can see that this is bounded as the glued in Eguchi-Hanson parts are uniformly bounded and the orbifold has some constant diameter, so the only bit not considered is the gluing region, but the diameter here, from our construction, is also clearly bounded for small $a$, so the diameter is bounded. Now considering the Ricci curvature, this follows from our construction of $\omega_0$ as we have that this is Ricci flat everywhere except the gluing region, and in this region it is $\omega_{EH,a}^2-\omega_{\text{flat}}^2=\mathcal{O}(a^4)$, from this it then follows that Ricci curvature is bounded for small $a$ from the following calculation:
\begin{align}
    \log(\omega_{EH,a}^2)=\log(\omega^2_{\text{flat}}(1+\mathcal{O}(a^4))=\log(\omega_{\text{flat}}^2)+\mathcal{O}(a^4)
\end{align}
Now to look at the $a^\alpha\|u\|_{C^{2,\alpha}}$ term, by considering the Schauder estimate given earlier, \ref{schauder}, on balls $B_1$ and $B_2$ we have:
\begin{align}
     \|u\|_{C^{2,\alpha}(B_1)}\leq C(\|f\|_{C^{0,\alpha}(B_2)}+\|u\|_{L^p(B_2)})
\end{align}
We want to rescale the balls $B_1$ and $B_2$ by a factor of $r$. To do this, we will look at how each term rescales under $x_r(q)=x(rq)$, $q\in B_k$:
\begin{align}
    \begin{split}
    \|x_r\|_{C^{2,\alpha}(B_1)}&=\|x_r\|_{C^2(B_1)}+[\nabla^2x_r]_\alpha=r^2\|x\|_{C^2(B_r)}+r^{2+\alpha}[\nabla^2x]_\alpha\\
    &\gtrsim r^{2+\alpha}\|x\|_{C^{2,\alpha}(B_r)}\\
    \|\Delta x_r\|_{C^{0,\alpha}(B_2)}&=\|\Delta x_r\|_{C^0(B_2)}+[\Delta x_r]_\alpha\lesssim r^2\|\Delta x\|_{C^0(B_{2r})}+r^{2+\alpha}[\Delta x]_\alpha\\
    &\lesssim r^2\|\Delta x\|_{C^{0,\alpha}(B_{2r})}\\
    \|x_r\|^p_{L^p(B_2)}&= r^{-n}\|x\|_{L^p(B_{2r})}^p, \implies \|x_r\|_{L^p(B_2)}= r^{-n/p}\|x\|_{L^p(B_{2r})}
    \end{split}
\end{align}
Note in our second inequality that the last inequality follows from choosing small enough $r$. From this we have the rescaled Schauder estimate:
\begin{align}\label{estBr}
    \|u\|_{C^{2,\alpha}(B_r)}\lesssim r^{-(2+\alpha)}(r^2\|\Delta u\|_{C^{0,\alpha}(B_{2r})}+r^{-n/p}\|u\|_{L^p(B_{2r})})
\end{align}
We now need to expand this over the Kummer surface, $Z$, we do this by using harmonic coordinates, see \cite[p.298]{JoyceBook}, we can cover $Z$ in balls of radius $r=\mathcal{O}(a)$, which are close to Euclidean. On each of these balls we can then apply our estimate, \ref{estBr}, this is analogue of \cite[Theorem G1, p.295]{JoyceBook}. Note, above we used $B_r$ to denote Euclidean balls of radius $r$ and below we will use $B_r(p)$ to denote almost Euclidean balls of radius $r$ on $Z$. We also choose $p$ below, such that $\|.\|_{C^{2,\alpha}(Z)}=\|.\|_{C^{2,\alpha}(B_r(p))}$. 
\begin{align}
\begin{split}
    \|u\|_{C^{2,\alpha}(Z)}&=\|u\|_{C^{2,\alpha}(B_r(p))}\lesssim r^{-\alpha}\|\Delta u\|_{C^{0,\alpha}(B_{2r}(p))}+r^{-(2+\alpha+n/p)}\|u\|_{L^p(B_{2r}(p))}\\&\lesssim r^{-\alpha}\|\Delta u\|_{C^{0,\alpha}(B_{2r}(p))}+r^{-(2+\alpha+n/p)}\|u\|_{L^p(Z)}
    \end{split}
\end{align}
By covering the ball $B_{2r}(p)$ by $N$ balls of radius $r$ centred at $p_i$ we have:
\begin{align}
    \begin{split}
    \|u\|_{C^{2,\alpha}(Z)}&\lesssim r^{-\alpha}\sum_{i=1}^N\|\Delta u\|_{C^{0,\alpha}(B_{r}(p_i))}+r^{-(2+\alpha+n/p)}\|u\|_{L^p(Z)}\\&\lesssim Nr^{-\alpha}\|\Delta u\|_{C^{0,\alpha}(B_r(p))}+r^{-(2+\alpha+n/p)}\|u\|_{L^p(Z)}
    \\&\lesssim a^{-\alpha}\|\Delta u\|_{C^{0,\alpha}(Z)}+a^{-(2+\alpha+n/p)}\|u\|_{L^p(Z)}
    \end{split}
\end{align}
Observe that $N$ is independent of both $r$ and $a$ and the last relationship comes from $r=\mathcal{O}(a)$. It then follows, by combining the above inequalities, that:
\begin{align}
    \begin{split}
   a^\alpha \|u\|_{C^{2,\alpha}(Z)}&\lesssim \|\Delta u\|_{C^{0,\alpha}(Z)}+a^{-(2+n/p)}\|u\|_{L^p(Z)}\lesssim\|\Delta u\|_{C^{0,\alpha}(Z)}+a^{\epsilon-4}\|u\|_{L^2_2(Z)}\\
   &\lesssim \|\Delta u\|_{C^{0,\alpha}(Z)}+\|\Delta u\|_Y\lesssim \|\Delta u\|_Y
   \end{split}
\end{align}
For the second relationship, we have used $\epsilon=4-2-n/p$ and taking $n=4$ and $p=6$ we have that $L^2_2\hookrightarrow L^p$. Note that the choice for the value of $p$ here is not unique and by choosing a different value for $p$ we get a different range of possible values for $\alpha$. For the third relationship, we have used \ref{L22}. The embedding $L^2_2\hookrightarrow L^p$ is not trivial in this setting and could depend on $a$. However, using that the Ricci curvature, diameter and volume are bounded for small $a$ we can apply \cite[Theorem 4.2, p9]{Hebey1996} and our estimate, \ref{estBr}, to see that this result also holds in our setting. So we have:
\begin{align}
    \|u\|_X=a^{\epsilon-4}\|u\|_{L^2_2(Z)}+a^\alpha\|u\|_{C^{2,\alpha}(Z)}\lesssim\|\Delta u\|_Y+\|\Delta u\|_Y\lesssim\|\Delta u\|_Y
\end{align}
which is the desired result. \qed

\begin{prop} \label{quadratic}
    We have the following estimate for small $a$:
    \begin{align}
        \|Qu_1-Qu_2\|_Y\lesssim a^{-2\alpha}\|u_1-u_2\|_X\|u_1+u_2\|_X
    \end{align}
\end{prop}
\textbf{Proof:} To begin observe the following, which will be useful going forward:
\begin{align} \label{longexpansion}
    \begin{split}
    \|u_1-u_2\|_X\|u_1+u_2\|_X&=(a^{-4+\epsilon})^2\|u_1-u_2\|_{L^2_2(Z)}\|u_1+u_2\|_{L^2_2(Z)}\\&+a^{-4+\epsilon+\alpha}(\|u_1-u_2\|_{L^2_2(Z)}\|u_1+u_2\|_{C^{2,\alpha}(Z)}+\|u_1+u_2\|_{L^2_2(Z)}\|u_1-u_2\|_{C^{2,\alpha}(Z)})\\
    &+a^{2\alpha}\|u_1-u_2\|_{C^{2,\alpha}(Z)}\|u_1+u_2\|_{C^{2,\alpha}(Z)}
    \end{split}
\end{align}
Now recalling our definition for our quadratic operator, $Q$, we have the following:
\begin{align}
    \|Qu_1-Qu_2\|_Y=\|((\mathcal{D}u_1)^2-(\mathcal{D}
u_2)^2)/\omega_0^2\|_Y=\|((\mathcal{D}u_1-\mathcal{D}u_2)(\mathcal{D}u_1+\mathcal{D}u_2))/\omega_0^2\|_Y
\end{align}
Now for simplicity we are going to omit the $\omega_0$ as this makes no material difference in the analysis. So we have:
\begin{align}
    \|Qu_1-Qu_2\|_Y=a^{-4+\epsilon}\|(\mathcal{D}u_1-\mathcal{D}u_2)(\mathcal{D}u_1+\mathcal{D}u_2)\|_{L^2(Z)}+\|(\mathcal{D}u_1-\mathcal{D}u_2)(\mathcal{D}u_1+\mathcal{D}u_2)\|_{C^{0,\alpha}(Z)}=\text{I}+\text{II}
\end{align}
We will now consider I and II separately.
\begin{align}
    \begin{split}
    \text{I}&\leq t^{-4+\epsilon}\|\mathcal{D}u_1-\mathcal{D}u_2\|_{C^0(Z)}\|\mathcal{D}u_1+\mathcal{D}u_2\|_{L^2(Z)}=a^{-4+\epsilon}\|\mathcal{D}(u_1-u_2)\|_{C^0(Z)}\|\mathcal{D}(u_1+u_2)\|_{L^2(Z)}\\
    &=a^{-4+\epsilon}\|\nabla^2(u_1-u_2)\|_{C^0(Z)}\|\nabla^2(u_1+u_2)\|_{L^2(Z)}\leq a^{-4+\epsilon}\|u_1-u_2\|_{C^2(Z)}\|u_1+u_2\|_{L^2_2(Z)}\\
    &\leq a^{-\alpha}\|u_1-u_2\|_X\|u_1+u_2\|_X
    \end{split}
\end{align}
Here, the first inequality follows from the standard H\"older inequality and the last inequality is from noticing that this matches the third term in our expression earlier \ref{longexpansion}. Now to consider II:
\begin{align}
\begin{split}
    \text{II}&\leq \|\mathcal{D}u_1-\mathcal{D}u_1\|_{C^{0,\alpha}(Z)}\|\mathcal{D}u_1-\mathcal{D}u_2\|_{C^{0,\alpha}(Z)}\leq \|u_1-u_2\|_{C^{2,\alpha}(Z)}\|u_1+u_2\|_{C^{2,\alpha}(Z)}\\&\leq a^{-2\alpha}\|u_1-u_2\|_X\|u_1+u_2\|_X
    \end{split}
\end{align}
Here, the first inequality is from the definition and the last two are the same arguments as used in I, but this time we identify the match with the final term in \ref{longexpansion}. So, considering a sufficiently small $a$ we have the desired result. \qed

\section{A Calabi-Yau Metric}\label{4}
Now that we have estimates for $e_a$, $\Delta$ and $Q$ we can further reduce our problem to finding a fixed point of a suitable map, which represents the problem given earlier, \ref{CYequation0}. This reduces to showing that we can locally apply the Banach fixed point theorem, it is here that we have to apply our estimates. Once we have this, we have our local solution and so our Calabi-Yau metric on the Kummer surface.\\
\\

Now, considering our problem as a whole, begin by considering the following map, $F$, and equation for a fixed point of said map:
\begin{align}
    F=-e_a-Q\circ \Delta^{-1}\text{ , fixed point} \implies \psi=-e_a-Q\circ \Delta^{-1}(\psi)
\end{align}
Now let $\phi=\Delta^{-1}\psi$:
\begin{align}
    \Delta(\phi)=-e_a-Q\circ \Delta^{-1}(\Delta(\phi))
\end{align}
Which is exactly our initial problem. So we need to find the fixed points, $\psi$, of $F$ and then by applying $L^{-1}$ we have our solutions, $\phi$. To show this, we are going to apply the Banach Fixed Point theorem.

\begin{theorem}
    (Banach Fixed Point theorem) Let $(X, d)$ be a non-empty complete metric space: and let $F:X\rightarrow X$ be a contraction,  i.e:
    \begin{align}
        d(F(x_1), F(x_2))\leq kd(x_1, x_2) \text{, for } x_1, x_2\in X\text{ and } k<1.
    \end{align}
    Then $F$ has a unique fixed point, $x=F(x)$. \cite[p.138]{Sobolev}
\end{theorem}

\begin{theorem}
    There exists $R>0$ and $B_R=\{u\in Y:\|\psi\|_Y\leq R\}$, such that for small $a$, $F:Y\rightarrow Y$ has a unique fixed point $\psi$ in $B_R$.
\end{theorem}

\textbf{Proof:}
To show this, we need to show that $F$ is a contraction on $B_R$ and that $F$ maps $B_R$ into $B_R$. Begin by setting $R=a^{\epsilon/2}$ and restricting $F$ to this ball. To show that this is is a contraction consider the following, for $u_1, u_2\in B_R$:
\begin{align}
    \begin{split}
    \|F(u_1)-F(u_2)\|_Y&=\|-e_a-Q\circ \Delta^{-1}u_1+e_a+Q\circ \Delta^{-1}u_2\|_Y=\|Q\circ \Delta^{-1}u_1-Q\circ \Delta^{-1}u_2\|_Y\\
    &\leq a^{-2\alpha}\|\Delta^{-1}u_1-\Delta^{-1}u_2\|_X\|\Delta^{-1}u_1+\Delta^{-1}u_2\|_X\leq  a^{-2\alpha}\|u_1-u_2\|_Y\|u_1+u_2\|_Y\\
    &\leq 2a^{-2\alpha}R\|u_1-u_2\|_Y=2a^{-2\alpha+\epsilon/2}\|u_1-u_2\|_Y
    \end{split}
\end{align}
By taking $\alpha<1/3$ we have $\epsilon/2-2\alpha=1-1/3-2\alpha>0$, so for small $a$ we have $2a^{\epsilon/2-2\alpha}<1$, so $F$ is a contraction. Now to show $F:B_R\rightarrow B_R$ consider the following for $u\in B_R$:
\begin{align}
    \begin{split}
    \|F(u)\|_Y&=\|-e_a-Q\circ \Delta^{-1}u\|_Y\leq \|e_a\|_Y+\|Q\circ \Delta^{-1}u\|_Y\leq ca^{\epsilon}+\|Q\circ \Delta^{-1}u\|_Y\\
    &\leq c a^{\epsilon}+a^{-2\alpha}\|\Delta^{-1}u\|_X^2\leq ca^\epsilon+a^{-2\alpha}\|u\|^2_Y\leq c a^\epsilon+a^{-2\alpha}R^2=ca^\epsilon +a^{\epsilon-2\alpha}
    \end{split}
\end{align}
As this is bounded above by $a^{\epsilon/2}$ for small enough $a$ we have that $F$ maps $B_R$ into $B_R$ as desired. We can now apply the Banach fixed point theorem to see that $F$ has a unique fixed point in $B_R$ for small $a$. \qed

\begin{cor}
    The unique fixed point $\psi$ corresponds to a local solution, $\phi$, of our Calabi-Yau equation \ref{CYequation0} and so we have a Calabi-Yau metric, $\omega=\omega_0+\mathcal{D}\phi$, on the Kummer surface. 
\end{cor}

Now that we have that a Calabi-Yau metric exists locally and is locally unique, it is worth demonstrating that there is also a fairly standard way to show that if there exists a global Calabi-Yau metric it must be unique. We also have the Calabi-Yau theorem \cite[p.98]{JoyceBook};\cite[p.13]{ExtKah} which, in the Ricci flat case $c_1(M)=0$, gives us both global uniqueness and global existence. The significance of the above work is in the analytic structure and the explicit information given about the metrics, none of which is given by the Calabi-Yau theorem. As showing global uniqueness given existence is not particularly long, we have included the result below:

\begin{prop}
    On a compact K\"ahler manifold, $M$, with complex dimension $n$,  there exists at most one metric in each K\"ahler class $[\omega]\in H^{1,1}(M, \mathbf{R})$ such that $\text{Ric}(\omega)=0$. \cite[p.50]{ExtKah}
\end{prop}

\textbf{Proof:} Begin by assuming that there are two different Ricci flat K\"ahler metrics in the same class, $\omega, \omega_0$,  they differ by $i\partial\bar\partial\varphi$, for a smooth $\varphi$,  so $\omega=\omega_0+i\partial\bar\partial\varphi$ \cite[p.80]{JoyceBook};\cite[p.8]{ExtKah}. Recalling earlier, we showed that $\text{Ric}(\omega)=-i\partial\bar\partial\log\det(g_{i\bar{j}})$ and observe that $\omega^n=\det(g_{i\bar{j}})(idz_1\wedge d\bar{z}_1)\wedge...\wedge(idz_n\wedge d\bar{z}_n)$ and similarly for $\omega_0$ with associated metric $g^0_{i\bar{j}}$. This then gives us the following relationships:
\begin{align}
    \begin{split}0&=\partial\bar\partial\log\omega^n=\partial\bar\partial\log\omega_0^n \\
    0&=\partial\bar\partial\log \omega^n-\partial\bar\partial\log \omega_0^n=\partial\bar\partial\log\left(\frac{\omega^n}{\omega_0^n}\right)
    \end{split}
\end{align}
This implies that $\frac{\omega^n}{\omega_0^n}$ is harmonic and as $M$ is compact this means that $\frac{\omega^n}{\omega_0^n}$ must be constant. As we know both $\omega$ and $\omega_0$ are in the same K\"ahler class, we have that:
\begin{align}
    \int_M\omega_0^n=\int_M\omega^n=\int_M(\omega_0+i\partial\bar\partial\varphi)^n
\end{align}
So we can conclude that $\varphi=0$ and so we have that $\omega=\omega_0$. \qed

This clearly shows us that if there exists a global Calabi-Yau metric on the Kummer surface then it is unique.

\bibliography{bibliography.bib}

@misc{Simpap,
      title={Calabi-{Y}au metrics on {K}ummer surfaces as a model glueing problem}, 
      author={Simon Donaldson},
      year={2010},
      eprint={1007.4218},
      archivePrefix={arXiv},
      primaryClass={math.DG},
      url={https://arxiv.org/abs/1007.4218}, 
}

@book{Cman,
author = {Huybrechts, Daniel.},
address = {Berlin, Heidelberg},
booktitle = {Complex Geometry: An Introduction},
edition = {1st ed. 2005.},
isbn = {3-540-26687-9},
language = {eng},
publisher = {Springer Berlin Heidelberg},
series = {Universitext},
title = {Complex Geometry : An Introduction },
year = {2005},
}

@article{EHmetric,
title = {Asymptotically flat self-dual solutions to euclidean gravity},
journal = {Physics Letters B},
volume = {74},
number = {3},
pages = {249-251},
year = {1978},
issn = {0370-2693},
doi = {https://doi.org/10.1016/0370-2693(78)90566-X},
url = {https://www.sciencedirect.com/science/article/pii/037026937890566X},
author = {Tohru Eguchi and Andrew J. Hanson},

}

@misc{introCY,
      title={An Introduction to {C}alabi-{Y}au Manifolds}, 
      author={Aidan Patterson},
      year={2025},
      eprint={2502.00167},
      archivePrefix={arXiv},
      primaryClass={math.DG},
      url={https://arxiv.org/abs/2502.00167}, 
}

@article{Joyce2021,
   title={A new construction of compact torsion-free {$G_2$}-manifolds by gluing families of {E}guchi–{H}anson spaces},
   volume={117},
   ISSN={0022-040X},
   url={http://dx.doi.org/10.4310/jdg/1612975017},
   DOI={10.4310/jdg/1612975017},
   number={2},
   journal={Journal of Differential Geometry},
   publisher={International Press of Boston},
   author={Joyce, Dominic and Karigiannis, Spiro},
   year={2021},
   month=feb }

@article{GH1979,
  author       = {G. Gibbons and S. Hawking},
  title        = {Classification of Gravitational Instanton Symmetries},
  journal      = {Communications in Mathematical Physics},
  volume       = {66},
  number       = {3},
  pages        = {291--310},
  year         = {1979},
  doi          = {10.1007/BF01197189},
  url          = {https://doi.org/10.1007/BF01197189}
}

@book{doca,
author = {Carmo, Manfredo Perdig{\~a}o do.},
address = {Boston},
booktitle = {Riemannian geometry},
isbn = {0817634908},
language = {eng},
lccn = {91037377},
publisher = {Birkha\"auser},
series = {Mathematics, theory \& applications},
title = {Riemannian geometry },
year = {1992},
}

@book{Sobolev,
  author    = {Haim Brezis},
  title     = {Functional Analysis, Sobolev Spaces and Partial Differential Equations},
  series    = {Universitext},
  publisher = {Springer},
  address   = {New York, NY},
  edition   = {1},
  year      = {2010},
  isbn      = {978-0-387-70913-0}
}

@book{JoyceBook,
  author    = {Joyce, Dominic D.},
  title     = {Compact Manifolds with Special Holonomy},
  publisher = {Oxford University Press},
  year      = {2000},
  series    = {Oxford Mathematical Monographs},
  address   = {Oxford}
}

@inbook{CYric, place={Cambridge},
series={London Mathematical Society Student Texts}, 
title={The Ricci form of K\"ahler manifolds},
booktitle={Lectures on Kähler Geometry}, 
publisher={Cambridge University Press}, 
author={Moroianu, Andrei}, 
year={2007}, 
pages={119–124}, collection={London Mathematical Society Student Texts}}

@book{ExtKah,
author = {Szekelyhidi, Gabor},
address = {Providence, Rhode Island},
title = {An introduction to extremal {K}{\"a}hler metrics},
isbn = {1-4704-1687-5},
language = {eng},
publisher = {American Mathematical Society},
series = {Graduate Studies in Mathematics ; Volume 152},
year = {2014 - 2014},
}

@book{bochner,
author = {Petersen, Peter},
address = {Cham},
booktitle = {Riemannian Geometry},
edition = {3rd ed. 2016.},
isbn = {3-319-26654-3},
language = {eng},
publisher = {Springer International Publishing},
series = {Graduate Texts in Mathematics, 171},
title = {Riemannian Geometry },
year = {2016},
}

@article{LiYau,
author = {Li, Peter and Yau, Shing},
year = {1980},
month = {01},
pages = {},
title = {Estimates of eigenvalues of a compact {R}iemannian manifold},
volume = {36},
isbn = {9780821814390},
journal = {Proc. Sympos. PureMath.},
doi = {10.1090/pspum/036/573435}
}

@book{CZestimate,
year = {1998},
publisher = {Springer},
author = {Gilbarg, David and Trudinger, Neil S},
booktitle = {Elliptic partial differential equations of second order},
series = {Grundlehren der mathematischen Wissenschaften, 224},
edition = {2nd ed., rev. 3rd printing.},
isbn = {354013025X},
language = {eng},
lccn = {97048777},
title = {Elliptic partial differential equations of second order },
}

@article{Yau1978,
  author  = {Yau, Shing-Tung},
  title   = {On the {R}icci curvature of a compact {Kä}hler manifold and the complex {M}onge-{A}mp{\`e}re equation},
  journal = {Communications on Pure and Applied Mathematics},
  volume  = {31},
  number  = {3},
  pages   = {339--411},
  year    = {1978}
}

@misc{jiang202,
      title={The {K}ummer Construction of {C}alabi-{Y}au and Hyper-{K}\"{a}hler Metrics on the {$K3$} Surface, and Large Families of Volume Non-collapsed Limiting Compact Hyper-{K}\"{a}hler Orbifolds}, 
      author={Thomas Jiang},
      year={2025},
      eprint={2501.08491},
      archivePrefix={arXiv},
      primaryClass={math.DG},
      url={https://arxiv.org/abs/2501.08491}, 
}

@book{Hebey1996,
  author    = {Emmanuel Hebey},
  title     = {Sobolev Spaces on {R}iemannian Manifolds},
  series    = {Lecture Notes in Mathematics},
  volume    = {1635},
  publisher = {Springer},
  address   = {Berlin, Heidelberg},
  year      = {1996},
  isbn      = {978-3-540-61722-8},
  doi       = {10.1007/BFb0092907}
}

@book{Ballmann,
  author    = {Werner Ballmann},
  title     = {Lectures on {K}\"{a}hler Manifolds},
  series    = {ESI Lectures in Mathematics and Physics},
  publisher = {European Mathematical Society},
  year      = {2006},
 
}

@misc{Platt,

  author = {Douglas, Michael R. and G{\'o}mez-Serrano, Javier and Lehmann, Fabian and Platt, Daniel and Qi, Yidi and Tong, Freid},

  title  = {Explicit elliptic estimates in geometric analysis},

  note   = {To appear}

}

@article{Topiwala,
author = {Topiwala, P.},
journal = {Inventiones mathematicae},
pages = {425-448},
title = {A new proof of the existence of Kähler-Einstein metrics on K3, I.},
url = {http://eudml.org/doc/143489},
volume = {89},
year = {1987},
}

\end{document}